\documentclass[11pt,reqno]{amsart}

\allowdisplaybreaks[2]

\usepackage{amsmath}
\usepackage{amssymb}
\usepackage{amsthm}
\usepackage{stackrel}
\input{xy}
\xyoption{all}
\usepackage{bbm}
\usepackage[shortlabels]{enumitem}
\usepackage[hmargin=2cm,vmargin=1.8cm]{geometry}
\usepackage[colorlinks=true,linkcolor=blue,citecolor=blue]{hyperref}

\newcommand{\ul}{\underline}
\newcommand{\rad}{\textup{rad}\,}
\newcommand{\Hom}{\textup{Hom}}
\newcommand{\Ext}{\textup{Ext}}
\newcommand{\End}{\textup{End}}
\newcommand{\card}{\textup{card}}
\newcommand{\spn}{\textup{span}\,}
\newcommand{\HH} {\operatorname{HH}\nolimits}
\newcommand{\Der}{\textup{Der}}
\newcommand{\Inn}{\textup{Inn}}

\newcommand{\calc}{\mathcal{C}}
\newcommand{\hh}{\operatorname{H}\nolimits}
\newcommand{\pgq}{\geqslant}
\newcommand{\ppq}{\leqslant}
\newcommand{\cale}{\mathcal{E}}
\renewcommand{\ker}{\operatorname{Ker}\nolimits}
\newcommand{\im}{\operatorname{Im}\nolimits}
\newcommand{\id}{\operatorname{id}\nolimits}
\newcommand{\mo}{\mathfrak{s}}
\newcommand{\mt}{\mathfrak{t}}
\newcommand{\tup}{\textsuperscript}
\newcommand{\sfp}{\mathsf{P}}
\newcommand{\sfq}{\mathsf{Q}}
\newcommand{\sfr}{\mathsf{R}}
\newcommand{\sfbar}{\mathsf{Bar}}
\newcommand{\da}{\text{-}}

\newtheorem{thm}{Theorem}[section]
\newtheorem{prop}[thm]{Proposition}
\newtheorem{lem}[thm]{Lemma}
\newtheorem{cor}[thm]{Corollary}
\newtheorem{thmIntro}{Theorem}    

\theoremstyle{definition}
\newtheorem{defn}[thm]{Definition}
\newtheorem{example}[thm]{Example}

\theoremstyle{remark}

\newtheorem{remark}[thm]{Remark}

\numberwithin{equation}{section}

\begin{document}


\title{Hochschild cohomology of relation extension algebras}
\dedicatory{Dedicated to Eduardo N. Marcos for his 60th birthday}
\author[I. Assem]{Ibrahim Assem}
\address{I. Assem, D\'epartement de math\'ematiques, Universit\'{e} de Sherbrooke, Sherbrooke, Qu\'{e}bec, Canada, JIK2R1. }
\email{Ibrahim.Assem@usherbrooke.ca}
\author[M. A. Gatica]{Maria Andrea Gatica}
\address{ M. A. Gatica, Departamento de Matem\'atica, Universidad Nacional del Sur, Avenida Alem 1253, (8000) Bah\'{\i}a Blanca, Buenos Aires, Argentina.}
\email{mariaandrea.gatica@gmail.com}
\author[R. Schiffler]{Ralf Schiffler}
\address{R. Schiffler, Department of Mathematics, University of Connecticut, 
Storrs, CT 06269-3009, USA}
\email{schiffler@math.uconn.edu}
\author[R. Taillefer]{Rachel Taillefer\textup{*}}
\thanks{* corresponding author}
\address{R. Taillefer, Laboratoire de Math\'ematiques (UMR 6620), Universit\'e Blaise Pascal, Complexe universitaire des C\'ezeaux, 3, place Vasarely -- TSA 60026 -- CS 60026, 63178 Aubi\`ere Cedex, France.}
\email[corresponding author]{Rachel.Taillefer@math.univ-bpclermont.fr}



\begin{abstract} Let  $B$ be the split extension of a finite dimensional algebra $C$ by a $C$-$C$-bimodule $E$.  We define a morphism of associative graded algebras $\varphi^*:\HH^*(B)\rightarrow \HH^*(C)$ from the Hochschild cohomology of $B$ to that of $C$, extending similar constructions for the first cohomology groups made and studied by Assem, Bustamante, Igusa, Redondo and Schiffler.

 In the case of a trivial extension $B=C\ltimes E$, we give necessary and sufficient conditions for each $\varphi^n$ to be surjective.  We prove the surjectivity of $\varphi^1$ for a class of trivial extensions that includes relation extensions and hence cluster-tilted algebras. Finally, we study the kernel of $\varphi^1$ for any trivial extension, and give a more precise description of this kernel in the case of relation extensions.
\end{abstract}

\maketitle

\section*{Introduction}

The purpose of this paper is to study the relation between the first Hochschild cohomology group of an algebra $C$ with coefficients in the regular bimodule $_CC_C$, and that of a relation extension $B$ of $C$. We recall that, if $k$ is a field, $C$ a finite dimensional $k$-algebra and $E$ a $C$-$C$-bimodule, equipped with a product $E \otimes_C E \rightarrow E$, then the vector space $B=C \oplus E$ is given an algebra structure if we set
\[  (c,x)(c',x')=(cc',cx'+xc'+xx') \] 
for $(c,x), \ (c',x') \in C \oplus E$. The algebra $B$ is then called a \emph{split extension} of $C$ by $E$. If $E^2=0$, then $B$ is called a \emph{trivial extension} of $C$ by $E$, and if moreover $C$ is triangular of global dimension at most two and $E= \Ext^2_C(DC,C)$, then $B$ is called the \emph{relation extension} of $C$, see \cite{ABS1}.

Relation extension algebras are of interest because of  their close relationship with cluster-tilted algebras. Cluster-tilted algebras were defined in \cite{BMR} (and, independently, in \cite{CCS} for  type $\mathbb{A}$) as a by-product of the now extensive theory of cluster algebras of S.~Fomin and A.~Zelevinsky, see, for instance, \cite{FZ}. They were the subject of several investigations. In particular, it was proved in \cite{ABS1} that, if $C$ is a tilted algebra, then its relation extension is cluster-tilted and, conversely, every cluster-tilted algebra is of this form. Of course, there exist relation extensions which are not cluster-tilted algebras. The first Hochschild cohomology groups of a tilted algebra $C$ and the corresponding cluster-tilted algebra $B$ were compared by means of a linear map  $\varphi\colon \HH^1(B)\rightarrow \HH^1(C)$, see \cite{AR,ABIS,ARS}. In each of these papers, it appeared that $\varphi$ is surjective, but the proof in each case was long and combinatorial.

The present paper arose from an attempt to produce a purely homological proof of the above mentioned results. We start by proving that, for every split extension $B$ of $C$ and every $n \geq 1$, there exists a linear map $\varphi^n: \HH^n(B)\rightarrow \HH^n(C)$ such that  $\varphi^1$ coincides with the map $\varphi\colon \HH^1(B)\rightarrow \HH^1(C)$ defined above. Recalling that each of  $\HH^*(B)=\bigoplus_{n\pgq 0}\HH^n(B)$ and $\HH^*(C)$ has an algebra structure given by the cup product, this leads to our first theorem.

\begin{thmIntro}\label{thmIntro:phi surj and ker} Let $B$ be a split extension of $C$, then the maps $\varphi^n$ induce a morphism of associative graded algebras $\varphi^*: \HH^*(B) \rightarrow \HH^*(C)$.
\end{thmIntro}

The algebra morphism $\varphi^*$ is called the  \emph{Hochschild projection morphism}. It is to be noted that $\varphi^*$ is not in general a morphism of graded Lie algebras and it is not surjective in general, even in the case where $B$ is a trivial extension of $C$, see Examples \ref{ex explicit ses} and \ref{not surjective}, respectively.

Using the results of \cite{CMRS}, we then proceed to find necessary and sufficient conditions for the surjectivity of each $\varphi^n$, see Proposition~\ref{prop: SC for surjectivity} and Corollary~\ref{cor:technical condition}. These conditions are satisfied, for every $n$, in the two familiar cases where $E=\,_CDC_C$, the minimal injective cogenerator bimodule, and $E=\,_CC_C$, the regular bimodule.

We then consider the $C$-$C$-bimodule $E_m=\Ext^m_C(DC,C)$ for any $m$. This bimodule appeared naturally in the study of a generalisation of the cluster-tilted algebras, namely the  \emph{$(m-1)$-cluster-tilted algebras}, introduced by H.~Thomas in \cite{T}. Indeed, it is shown in \cite{FPT} that, under some conditions, an $(m-1)$-cluster-tilted algebra can be written as the trivial extension of an iterated tilted algebra $C$ of global dimension at most $m$ by the bimodule $E_m$. We prove here that, if $B$ is the trivial extension  of any algebra $C$ by the $C$-$C$-bimodule $E_m$, then the first Hochschild projection morphism $\varphi^1: \HH^1(B)\rightarrow \HH^1(C)$ is surjective, see Theorem~\ref{cor:phi surjective Es}.

We next assume that $C$ is triangular of global dimension at most two and that  $E=E_2= \Ext_C^2(DC,C)$, so that the trivial extension $B$ of $C$ by $E$ is the relation extension of $C$. We then get our second main result.

\begin{thmIntro}\label{thmIntro:case triangular} Let $B$ be the relation extension of a triangular algebra $C$ of global dimension at most two by the $C$-$C$-bimodule $E=\Ext^2_C(DC,C)$. Then we have short exact sequences
\begin{enumerate}[(a)]
\item $0\rightarrow \hh^0(B,E)\longrightarrow \HH^0(B) \xrightarrow{\varphi^0} \HH^{0}(C)\rightarrow 0.$
\item $0\rightarrow \hh^1(B,E)\longrightarrow \HH^1(B) \xrightarrow{\varphi^1} \HH^{1}(C)\rightarrow 0.$
\end{enumerate}
\end{thmIntro}

This generalises to arbitrary relation extensions the main results of each of  \cite{AR,ABIS,ARS}, furnishing at the same time a homological interpretation of these results. As a nice consequence of this theorem, we get that, if $C$ is tilted, so that $B$ is cluster-tilted, then the Hochschild projection morphism $\varphi^*:\HH^*(B)\rightarrow\HH^*(C)$ is a surjective morphism of algebras, see Theorem~\ref{cor: HH*}.

The paper is organised as follows. After a short introductory section whose purpose is to fix the notation and recall a few useful facts, Section~\ref{sec:Hochschild proj} is devoted to the explicit construction of the Hochschild projection morphism and thus to the proof of our Theorem A. In  Section~\ref{sec:trivial extensions}, we specialise to trivial extensions and find necessary and sufficient conditions for the surjectivity of each of the $\varphi^n$. As a consequence, we prove that, if $B$ is the trivial extension of $C$ by $E$, then $\varphi^1$ is surjective. The study of the kernel of $\varphi^1$ in Section~\ref{sec:kernel} then leads us to our Theorem B in Section~\ref{sec:relation extensions}.

\section{Preliminaries}\label{sec:preliminaries}

\subsection{Algebras and quivers}

Throughout this paper, $k$ denotes a commutative field, and all algebras are finite dimensional over $k$ and have an identity.

Given a finite quiver $Q=(Q_0,Q_1)$, we denote by $kQ$ its path algebra. Two paths $w_1, w_2$ in $Q$ are called \emph{parallel} if they have the same source $\mo(w_1)=\mo(w_2)$ and the same target $\mt(w_1)=\mt(w_2)$.  A \emph{relation} from vertex $x$ to vertex $y$ is a linear combination 
$\sum_{i=1}^{t} \lambda_i w_i $ where $\lambda_i \in k \backslash \{ 0\}$ and the $w_i$ are distinct paths from $x$ to $y$. Let $kQ^+$ be the two-sided ideal of $kQ$ generated by the arrows. An ideal $I$ of $kQ$ is \emph{admissible} if there exists $s \pgq 2$ such that $(kQ^+)^s \subseteq  I  \subseteq (kQ^+)^2$. In this case, the pair $(Q,I)$ is called a \emph{bound quiver}. The algebra $C=kQ/I$ is basic, connected whenever $Q$ is, and finite dimensional. Moreover,  the ideal $I$ is finitely generated.  A \emph{system of relations} $R$ for $C$ is a subset of $\bigcup_{x,y \in (Q_C)_0}(e_xIe_y)$ such that $R$, but no proper subset of $R$, generates $I$ as a two-sided ideal.

Conversely, if $C$ is a finite dimensional basic and connected $k$-algebra, there exists a unique connected quiver $Q$ and (at least) an admissible ideal $I$ of $kQ$ such that $C \cong kQ/I$. The ordinary duality between right and left $C$-modules is denoted by $D=\Hom_k(-,k)$. Given a vertex $x$ in $Q$, we denote by $e_x$ the corresponding primitive idempotent of $C$ and by $P_C(x)=e_xC, \, I_C(x)=D(Ce_x)$ and $S_C(x)$ the corresponding  indecomposable right projective, injective and simple $C$-modules, respectively. For more details on bound quivers and their use in the representation theory of algebras, we refer the reader to \cite{ARS, ASS, S}.

\subsection{Hochschild cohomology}
\label{subsec:Hochschild cohomology}

Let $C$ be a finite dimensional $k$-algebra and $E$ be a $C$-$C$-bimodule which is finite dimensional over $k$. The  \emph{Hochschild complex} is the complex
\[ 0\rightarrow E \xrightarrow{b^1} \Hom_k(C,E)  \xrightarrow{b^2} \cdots  \rightarrow \Hom_k(C^{\otimes i},E)  \xrightarrow{b^{i+1}} \Hom_k(C^{\otimes {(i+1)}},E) \rightarrow \cdots \]
where, for each $i >0, \ C^{\otimes i}$ denotes the $i$-fold tensor product of $C$ with itself over $k$. The differentials $b^i$ are defined as follows: $ b^1: E \rightarrow \Hom_k(C,E)$ is given by $(b^1x)(c)=cx-xc$ for $x \in E, \ c \in C, $ and $b^{i+1}$ is given by 
\begin{align*}
 (b^{i+1}f)(c_0 \otimes \cdots \otimes c_{i})&= c_0f(c_1 \otimes \cdots \otimes c_{i}) + \sum_{j=1}^i (-1)^j f(c_0 \otimes \cdots \otimes c_{j-1} c_{j}\otimes \cdots  \otimes c_{i}) \\
&+  (-1)^{i+1}f(c_0 \otimes \cdots \otimes c_{i-1})c_{i}
\end{align*}
for a $k$-linear map $f: C^{\otimes i} \rightarrow E$ and elements $c_0, \cdots, c_{i} $ in $C$.

The $i$\tup{th} cohomology group of this complex is called the \emph{$i$\tup{th} Hochschild cohomology group} of $C$ with coefficients in $E$, and is denoted by $\hh^i(C,E)$. If $_CE_C={}_CC_C$ then we denote $\HH^i(C)= \hh^i(C,C)$.

We are particularly interested in the first Hochschild cohomology group.  Let $\Der(C,E)$ be the vector space of all $k$-linear maps $d: C\rightarrow E$ such that, for $c,c' \in C$, we have $d(cc')=cd(c') +d(c)c'$. Such a  map is called a \emph{derivation}. A derivation $d$ is \emph{inner} if there exists $x \in E$ such that $d=[x,-]$, that is, $d(c)=[x,c]=xc-cx$ for all $c \in C$. We denote by $\Inn(C,E)$ the subspace of $\Der(C,E)$ consisting of the inner derivations. Then we have $\hh^1(C,E) \cong \Der(C,E) / \Inn (C,E)$.

The latter expression may be simplified. Let $\{ e_1, \cdots, e_n\}$ be a complete set of primitive orthogonal idempotents of $C$. A derivation $d: C \rightarrow E$ is \emph{normalised} if $d(e_i)=0$ for all $i$. A normalised derivation $d$ has the nice property that, if $c \in e_iCe_j$, then $d(c) \in e_i E e_j$. Let $\Der_0(C,E)$ be the subspace of $\Der(C,E)$ consisting of the normalised derivations. Let also $ \Inn_0(C,E)= \Der_0(C,E) \cap \Inn(C,E)$. Then we have $\hh^1(C,E) \cong \Der_0(C,E)/ \Inn_0(C,E)$.

\subsubsection{Structure of $\HH^*(C)$.}\label{subsubsec:structure HH}

It is well-known that $\HH^*(C)=\bigoplus_{n\pgq 0}\HH^n(C)$ is endowed with some extra structure. It is an associative algebra for the cup-product  described as follows: if $\zeta_1\in\HH^s(C)$  and $\zeta_2\in \HH^{t}(C)$ are represented by cocycles $f_1\in\Hom_k(C^{\otimes s},C)$ and  $f_2\in\Hom_k(C^{\otimes t},C)$, then  $\zeta_1\smile \zeta_2$ is the cohomology class of  the map $f_1\times f_2\in \Hom_k(C^{\otimes (s+t)},C)$ defined by \[(f_1\times f_2)(c_1\otimes \cdots \otimes c_{s+t})=f_1(c_1\otimes\cdots \otimes c_s)f_2(c_{s+1}\otimes \cdots\otimes c_{s+t}).\] It was shown by Gerstenhaber in \cite{G} that this cup-product is graded commutative, that is, $\zeta_1\smile\zeta_2=(-1)^{st}\zeta_2\smile\zeta_1$, by means of special operations which he also used to construct a graded Lie product $[-,-]$ on $\HH^{*-1}(C).$ We shall not need this graded Lie bracket, but let us point out  that its restriction to $\HH^1(C)$ is  the natural Lie bracket on  derivations, given by $[d,d']=d\circ d'-d'\circ d.$

\subsubsection{One-point extensions.}

We need a result due to D.~Happel. We recall that, if $C$ is an algebra and $M$ a finitely generated right $C$-module, then the \emph{one-point extension} of $C$ by $M$ is the $k$-algebra
$$C[M]=\begin{pmatrix}C &0\\M&k\end{pmatrix}=\left\{ \begin{pmatrix}c &0\\m&\lambda\end{pmatrix} \, |  \, c \in C, m\in M, \lambda \in k \right\}$$ with the usual addition, and the multiplication induced from the $C$-module structure of $M$. We have the following result. 

\begin{thm}\cite[Theorem 5.3]{H}\label{thm 1.1} Let $B=C[M]$. Then there exists a long exact cohomology sequence
\[ 0\rightarrow  \HH^0(B) \rightarrow \HH^0(C)\rightarrow \End_CM/k \rightarrow   \HH^1(B) \rightarrow \HH^1(C)\rightarrow \Ext^1_C(M,M) \rightarrow \cdots\]
\end{thm}

For all unexplained notions and results about Hochschild cohomology, we refer the reader to \cite{CE, H, MJR, G}.

\section{The Hochschild projection morphisms}\label{sec:Hochschild proj}

Let $C$ be a finite-dimensional $k$-algebra and let $E$ be a finitely generated $C$-$C$-bimodule equipped with a product, that is, an associative $C$-$C$-bimodule morphism $E\otimes _C E\rightarrow E$. The \emph{split extension}  of $C$ by $E$ is the $k$-algebra $B$ which is equal to $C\oplus E$ as a $k$-vector space  and whose product is given by 
\[ (c,e)(c',e')=(cc',ce'+ec'+ee') \] for all $(c,e),\ (c',e')$ in $C\oplus E$. With this multiplication, $E$ becomes a two-sided ideal in $B$, hence a $B$-$B$-bimodule. If $E$ is a nilpotent ideal, that is, $E$ is contained in the radical of $B$, then we say that $B$ is a \emph{split-by-nilpotent extension} of $C$ by $E$. If $E^2=0$, then we say that $B$ is the \emph{trivial extension} of $C$ by $E$, which we denote by  $B=C\ltimes E.$

Given such a split extension $B$ of $C$ by $E$, there is an exact sequence of vector spaces
\[ 0\rightarrow E\xrightarrow{i} B\stackrel[q]{p}\rightleftarrows C\rightarrow 0 \] in which $p:(c,x)\mapsto c$ and $i:x\mapsto (0,x)$. Clearly, $p$ is an algebra morphism which has a section $q:c\mapsto (c,0)$. This sequence may also be viewed as a sequence of $B$-$B$-bimodules, where the $B$-$B$-bimodule structure of $C$ is defined by means of $p:B\rightarrow C$. Finally, it may also be viewed as a split exact sequence of $C$-$C$-bimodules by means of $q$.

Now let $d:B\rightarrow B$ be a $k$-linear map. Then $pdq:C\rightarrow C$ is also $k$-linear. It was shown in \cite[4.1]{AR} that this correspondence induces a map between the first Hochschild cohomology groups. The same morphism was studied in  \cite{ARS,ABIS}.  Here, we extend this to the $n$\tup{th} Hochschild cohomology groups. 

\begin{lem}\label{lem:prelim to definition of nth projection morphism}  Let $B$ be the split extension of $C$ by $E.$ Let $b_B^n$ (respectively $b_C^n$) denote the $n$\tup{th} differential in the Hochschild complex of $B$   with coefficients in $B$ (respectively of $C$ with coefficients in $C$). Then for any  $f\in \Hom_k(B^{\otimes n},B)$  we have
\[ b_C^{n+1}(pfq^{\otimes n})=pb_B^{n+1}(f)q^{\otimes (n+1)}. \]
\end{lem}

\begin{proof}
If $n=0,$ then $f\in \Hom_k(k,B)$ may be identified with an element in $B$. For any $c\in C$ we have
\begin{align*}
b_C(p(f))(c)=p(f)c-cp(f)=p(f)pq(c)-pq(c)p(f)=p(fq(c)-q(c)f)=p(b_B(f)(q(c)))=(pb_B(f)q)(c)
\end{align*} using the facts that $pq=\id_C$  and that $p$ is an algebra morphism.

Now assume that $n>0.$ For any $c_0\otimes\cdots\otimes c_n\in C^{\otimes(n+1)}$, using similar arguments as well as the fact that $q$ is also an algebra morphism, we have 
\begin{align*}
b_C^{n+1}(pfq^{\otimes n})(c_0\otimes\cdots\otimes c_n)&=c_0\, pfq^{\otimes n}(c_1\otimes\cdots\otimes  c_n)+\sum_{i=1}^{n}(-1)^ipfq^{\otimes n}(c_0\otimes\cdots\otimes c_{i-1}c_i\otimes\cdots\otimes c_n)\\
&\quad+(-1)^{n+1}pfq^{\otimes n}(c_0\otimes\cdots\otimes c_{n-1})\, c_n\\
&=pq(c_0)\, pf(q(c_1)\otimes\cdots\otimes  q(c_n))\\
&\quad+\sum_{i=1}^{n}(-1)^ipf(q(c_0)\otimes\cdots\otimes q(c_{i-1}c_i)\otimes\cdots\otimes q(c_n))\\
&\quad+(-1)^{n+1}pf(q(c_0)\otimes\cdots\otimes q(c_{n-1}))\, pq(c_n)\\
&=p\Big(q(c_0)\, f(q(c_1)\otimes\cdots\otimes  q(c_n))\\
&\quad+\sum_{i=1}^{n}(-1)^if(q(c_0)\otimes\cdots\otimes q(c_{i-1})q(c_i)\otimes\cdots\otimes q(c_n))\\
&\quad+(-1)^{n+1}f(q(c_0)\otimes\cdots\otimes q(c_{n-1}))\, q(c_n)\Big)\\
&=pb_B^{n+1}(f)q^{\otimes(n+1)}(c_0\otimes\cdots\otimes c_n).
\qedhere\end{align*}
\end{proof}

\begin{cor}\label{cor:definition of nth projection morphism} Let $B$ be the split extension of $C$ by $E.$ Then there exists a $k$-linear map $\varphi^n\colon \HH^n(B)\rightarrow \HH^n(C)$ given by $[f]\mapsto [pfq^{\otimes n}].$

The map $\varphi^n$ is called the \emph{$n$\tup{th} Hochschild projection morphism}. 
\end{cor}

\begin{proof}
Assume that $f\in\Hom_k(B^{\otimes n},B)$ is a cocycle, that is, that $b_B^{n+1}(f)=0$. Then the formula in Lemma~\ref{lem:prelim to definition of nth projection morphism} shows that $b_C^{n+1}(pfq^{\otimes n})=0$ so that $pfq^{\otimes n}$ is a cocycle.

Moreover, if $f$ is a coboundary, that is, $f=b_B^n(g)$ for some cochain $g\in\Hom_k(B^{\otimes (n-1)},B)$, then, using Lemma~\ref{lem:prelim to definition of nth projection morphism}, $pfq^{\otimes n}=pb_B^n(g)q^{\otimes n}=b_C^n(pgq^{\otimes(n-1)})$ is also a coboundary. Therefore we have a well-defined linear map $\varphi^n:\HH^n(B)\rightarrow \HH^n(C)$ which sends $[f]$ to $[pfq^{\otimes n}].$
\end{proof}

We now prove that the morphisms $\varphi^n$ induce a morphism of algebras from $\HH^*(B)$ to $\HH^*(C).$

\begin{thm} Considering $\HH^*(B)=\bigoplus_{n\pgq 0}\HH^n(B)$ and $\HH^*(C)$ as algebras with the cup-product, the maps $\varphi^n$ induce an algebra morphism $\varphi^*:\HH^*(B)\rightarrow\HH^*(C)$.
\end{thm}

\begin{proof}
Fix $\zeta_1\in\HH^s(B)$ and $\zeta_2\in\HH^t(B).$ Let $f_1\in\Hom_k(B^{\otimes s},B)$ (respectively $f_2\in\Hom_k(B^{\otimes t},B)$) be a cocycle representative of $\zeta_1$ (respectively of $\zeta_2$). Then $pf_1q^{\otimes s}$ (respectively $pf_2q^{\otimes t}$) is  a cocycle representative of $\varphi^s(\zeta_1)$ (respectively of $\varphi^t(\zeta_2)$).

Using the notation in paragraph \ref{subsubsec:structure HH}, we have 
\begin{align*}
\left((pf_1q^{\otimes s})\times (pf_2q^{\otimes t})\right)(c_1\otimes \cdots\otimes c_{s+t})&=pf_1(q(c_1)\otimes  \cdots\otimes q(c_s))pf_2(q(c_{s+1})\otimes \cdots\otimes q(c_{s+t}))\\
&=p\left(f_1(q(c_1)\otimes  \cdots\otimes q(c_s))f_2(q(c_{s+1})\otimes \cdots\otimes q(c_{s+t}))\right)\\
&=p(f_1\times f_2)(q(c_1)\otimes  \cdots\otimes q(c_{s+t})))\\
&=p(f_1\times f_2)q^{\otimes (s+t)}(c_1\otimes \cdots\otimes c_{s+t}).
\end{align*} Therefore $(pf_1q^{\otimes s})\times (pf_2q^{\otimes t})=p(f_1\times f_2)q^{\otimes(s+t)}$ and taking cohomology classes yields $\varphi^s(\zeta_1)\smile \varphi^t(\zeta_2)=\varphi^{s+t}(\zeta_1\smile\zeta_2).$ 

Finally, the identity element in $\HH^*(B)$ is $\mathbbm{1}_B$, the cohomology class of the map $u_B\in\Hom_k(k,B)$ which sends $1$ to $1$, and similarly for $\HH^*(C)$. We have 
\[ \varphi^0(\mathbbm{1}_B)=[pu_Bq^{\otimes 0}]=[u_C]=\mathbbm{1}_C.\qedhere \]
\end{proof}

\begin{remark}
In general, $\varphi^*$ is not a morphism of graded Lie algebras. Indeed, in the following example the Hochschild projection morphism $\varphi^1:\HH^1(B)\rightarrow\HH^1(C)$ is not a morphism of Lie algebras, so that $\varphi^*$ cannot be a morphism of graded Lie algebras. 
\end{remark}

\begin{example}\label{ex explicit ses}
Let $C$ be the Nakayama algebra of dimension $4$ whose quiver is \[\xymatrix@R=10pt{ 
0 \ar@/_/[rr]_{\alpha_0} &  & 1  \ar@/_5pt/[ll]_{\alpha_1}
}
\] bound by $\alpha_0\alpha_1=0=\alpha_1\alpha_0.$ Nakayama algebras have been much studied,
 and in particular their Hochschild cohomology was determined by K.~Erdmann and T.~Holm in \cite{EH}, where the algebra $C$ is denoted by $B_2^2.$

Let $B$ be the algebra of dimension $8$ whose quiver is  \[\xymatrix@R=10pt{ 
0 \ar@/_10pt/[rr]^{a_0} \ar@/_20pt/[rr]_{\bar{a}_1} &  & 1  \ar@/_10pt/[ll]^{a_1} \ar@/_20pt/[ll]_{\bar{a}_0}
}
\] bound by $a_0a_1=0,$ $a_1a_0=0$, $\bar{a}_0\bar{a}_1=0$, $\bar{a}_1\bar{a}_0=0$, $a_0\bar{a}_0=\bar{a}_1a_1$ and  $a_1\bar{a}_1=\bar{a}_0a_0.$ We shall denote the indices for the vertices and arrows in both algebras modulo $2$ so that, for instance, $a_2=a_0$.
The Hochschild cohomology of this algebra was studied  in \cite{ST}.

Now consider the morphism of algebras $p:B\rightarrow C$ determined by $p(e_i)=e_i$, $p(a_i)=\alpha_i$ and $p(\bar{a}_i)=(-1)^i\alpha_{i+1}.$ This map is surjective with corresponding section the algebra map $q:C\rightarrow B$  determined by $q(\alpha_i)=a_i$. Moreover, a straightforward computation shows that  $\ker p$ as a $k$-vector space has  basis $\{a_1-\bar{a}_0,a_0+\bar{a}_1,a_0\bar{a}_0,a_1\bar{a}_1\}$. Furthermore, using the relations in $B$ and the fact that all paths of length at least $3$ in $B$ vanish, it is immediate that $(\ker p)^2=0.$

Therefore we have a short exact sequence 
\[ 0\rightarrow E=\ker p\rightarrow B \stackrel[q]{p}\rightleftarrows C\rightarrow 0\] with $E^2=0$,  hence $B$ is isomorphic to the trivial extension $C\ltimes E$.
In fact,  $E\cong DC$ as $C$-$C$-bimodules, so that $B$ is the trivial extension  of $C$ by $DC$, the isomorphism $DC\rightarrow E$ being given by $\alpha_0^*\mapsto \bar{a}_0-a_1,$ $\alpha_1^*\mapsto \bar{a}_1+a_0$ and $e_i^*\mapsto a_i\bar{a}_i$ (where $\{e_i^*,\alpha_i^*; i=0,1\}$ denotes the dual basis of $DC$).

We now determine $\varphi^1.$ It was shown in \cite[Proposition 5.3]{EH} that $\dim_k\HH^1(C)=1$ and in \cite[Proposition 6.1]{ST} that  $\dim_k\HH^1(B)=4.$ Moreover, a basis of $\HH^1(B)$ was given in \cite[Proposition 6.2]{ST}, which in terms of derivations is the set $\{[u_0],[u_1],[v_0],[v_1]\}$ where 
\begin{enumerate}[$\triangleright$]
\item $u_0$ is the normalised derivation determined by $u_0(a_i)=a_i$ and $u_0(\bar a_i)=0$,
\item $u_1$ is the normalised  derivation determined by $u_1(\bar a_i)=(-1)^{i} a_{i+1}$ and $u_1(a_i)=0$,
\item $v_0$ is the normalised  derivation determined by $v_0(a_i)=(-1)^{i} \bar a_{i+1}$ and $v_0(\bar a_i)=0$,
\item $v_1$ is the normalised  derivation determined by $v_1(\bar a_i)=-\bar  a_{i}$ and $v_1(a_i)=0.$
\end{enumerate}

It is then easy to check that $\varphi^1([u_0])=[\xi]=-\varphi^1([v_0])$ and that $\varphi^1([u_1])=0=\varphi^1([v_1])$, where $\xi:C\rightarrow C$ is the normalised derivation determined by $\xi(\alpha_i)=\alpha_i$ for $i=1,2.$ 

Now consider the bracket $[u_0,v_0]=u_0v_0-v_0u_0.$ It sends $\bar{a}_i$ to $0$ and $a_i$ to $-v_0(a_i)$ so that $[u_0,v_0]=-v_0$. Therefore $\varphi^1([[u_0],[v_0]])=\varphi^1([[u_0,v_0]])=\varphi^1(-[v_0])=[\xi].$

On the other hand, $[\varphi^1([u_0]),\varphi^1([v_0])]=[[\xi],-[\xi]]=0.$ 

Therefore $\varphi^1([[u_0],[v_0]])\neq [\varphi^1([u_0]),\varphi^1([v_0])]$ and $\varphi^1$ is not a morphism of Lie algebras. 
\end{example}

Let $B$ be the split extension of $C$ by $E.$ We shall now compare the Hochschild cohomology groups of $B$ and $C$ by means of the Hochschild projection morphism. Applying the functor $\Hom_{B\da B}(B,-)$ to the short exact sequence of $B$-$B$-bimodules 
\[0\rightarrow E \xrightarrow{i} B \xrightarrow{p} C\rightarrow 0\]  yields a long exact cohomology sequence 
\begin{equation}\label{eq:long cohomology sequence}\begin{split}
0\rightarrow \hh^0(B,E)\rightarrow \HH^0(B)&\rightarrow \hh^0(B,C)\xrightarrow{\delta^0_B}\hh^1(B,E)\\\rightarrow \HH^1(B)&\rightarrow\hh^1(B,C)\xrightarrow{\delta^1_B} \hh^2(B,E)\rightarrow\HH^{2}(B)\rightarrow\cdots\\
\cdots&\rightarrow \hh^n(B,C)\xrightarrow{\delta^n_B} \hh^{n+1}(B,E)\rightarrow \HH^{n+1}(B)\rightarrow \hh^{n+1}(B,C)\xrightarrow{\delta^{n+1}_B}\cdots
\end{split}\end{equation}
where $\delta^n_B$ denotes the  $n$\tup{th} connecting morphism. Our first observation is that $\HH^n(C)$ embeds in $\hh^n(B,C)$ and that, if $n=0$, this embedding is an isomorphism. We begin with a lemma.

\begin{lem}\label{lemma:prelim cohom C embeds B}
Let $B$ be the split extension algebra of $C$ by a $C$-$C$-bimodule $E$.  Let $b_B^n$  (respectively $b_C^n$) be the differentials in the Hochschild complex of $B$ (respectively  $C$)  with coefficients in $C$. 
 Then, for any cochains $f\in\Hom_k(C^{\otimes n},C)$ and $g\in\Hom_k(B^{\otimes n},C)$, we have 
\[ b_B^{n+1}(fp^{\otimes n})=b_C^{n+1}(f)p^{\otimes (n+1)}\quad\text{ and }\quad b_C^{n+1}(gq^{\otimes n})=b_B^{n+1}(g)q^{\otimes (n+1)}. \]
\end{lem}

\begin{proof}
 Let us prove the first identity. For $n\pgq 1$ and for any $b_0\otimes\cdots\otimes b_n\in B^{\otimes (n+1)}$ we have
\begin{align*}
b_B^{n+1}(fp^{\otimes n})(b_0\otimes\cdots\otimes b_n)&=b_0\cdot fp^{\otimes n}(b_1\otimes\cdots\otimes  b_n)+\sum_{i=1}^{n}(-1)^ifp^{\otimes n}(b_0\otimes\cdots\otimes b_{i-1}b_i\otimes\cdots\otimes b_n)\\
&\quad+(-1)^{n+1}fp^{\otimes n}(b_0\otimes\cdots\otimes b_{n-1})\cdot b_n\\
&=p(b_0) f(p(b_1)\otimes\cdots\otimes  p(b_n))\\&\quad+\sum_{i=1}^{n}(-1)^if(p(b_0)\otimes\cdots\otimes p(b_{i-1})p(b_i)\otimes\cdots\otimes p(b_n))\\
&\quad+(-1)^{n+1}f(p(b_0)\otimes\cdots\otimes p(b_{n-1})) p(b_n)\\
&=b_C^{n+1}(f)p^{\otimes (n+1)}(b_0\otimes\cdots\otimes b_n)
\end{align*} using the facts that $p$ is an algebra morphism and  that $C$ is a $B$-$B$-bimodule via $p$. If $n=0$, then the cocycle $f$ may be viewed as an element in $C$, and for any $b_0\in B$ we have 
\[ b_B^1(f)(b_0)=b_0\cdot f-f\cdot b_0=p(b_0)f-fp(b_0)=b_C^1(f)(p(b_0)) \]
so that $b_B^1(f)=b_C^1(f)p.$
 
For the second identity,  we shall use the fact that if $c$ and $c'$ are in $C$, then $cc'=pq(c)c'=q(c)\cdot c'$ since $pq=\id_C$ and $B$ acts on $C$ via $p$. Similarly, $cc'=c\cdot q(c')$. Now take $c_0\otimes \cdots \otimes c_n\in C^{\otimes n}$ with $n\pgq 1$. Then
\begin{align*}
b_C^{n+1}(gq^{\otimes n})(c_0\otimes\cdots\otimes c_n)&=c_0 gq^{\otimes n}(c_1\otimes\cdots\otimes  c_n)+\sum_{i=1}^{n}(-1)^igq^{\otimes n}(c_0\otimes\cdots\otimes c_{i-1}c_i\otimes\cdots\otimes c_n)\\
&\quad+(-1)^{n+1}gq^{\otimes n}(c_0\otimes\cdots\otimes c_{n-1}) c_n\\
&=q(c_0)\cdot g(q(c_1)\otimes\cdots\otimes  q(c_n))\\
&\quad+\sum_{i=1}^{n}(-1)^ig(q(c_0)\otimes\cdots\otimes q(c_{i-1})q(c_i)\otimes\cdots\otimes q(c_n))\\
&\quad+(-1)^{n+1}g(q(c_0)\otimes\cdots\otimes q(c_{n-1}))\cdot q(c_n)\\
&=b_B^{n+1}(g)q^{\otimes(n+1)}(c_0\otimes\cdots\otimes c_n).
\end{align*} Here again, if $n=0$ then $g$ may be viewed as an element in $C$, and for any $c_0\in C$ we have 
\[ b_C^1(g)(c_0)=c_0g-gc_0=q(c_0)\cdot g-g\cdot q(c_0)=b_B^1(g)(q(c_0)) \] so that $b_C^1(g)=b_B^1(g)q.$
\end{proof}

\begin{lem}\label{cor:cohom C embeds B}
Let $B$ be the split extension algebra of $C$ by a $C$-$C$-bimodule $E$. Then
\begin{enumerate}[(a)]
\item There is an isomorphism $\HH^0(C)\cong\hh^0(B,C)$.
\item For every integer $n\pgq 0$, there is a monomorphism $\sigma_n:\HH^n(C)\hookrightarrow \hh^n(B,C)$ which sends $[f]$ to $[fp^{\otimes n}]$ with retraction $\nu_n:\hh^n(B,C)\twoheadrightarrow\HH^n(C)$ given by $[g]\mapsto [gq^{\otimes n}].$
\end{enumerate}
\end{lem}

\begin{proof} The relations in Lemma~\ref{lemma:prelim cohom C embeds B} show that  $\sigma_n$ and $\nu_n$ are well-defined maps. Clearly, we have $\nu_n\sigma_n=\id_{\HH^n(C)}$. Therefore $\sigma_n$ is injective with  retraction $\nu_n$ and we have proved (b).

Now to prove (a), we use the identity $\Hom_{k}(p,C)b_C^1=b_B^1$ from Lemma~\ref{lemma:prelim cohom C embeds B} (for $n=0$) and the fact that $\Hom_{k}(p,C):\End_k(C)\rightarrow \Hom_k(B,C)$ is injective: 
\[ \HH^0(C)=\ker b_C^1=\ker (\Hom_{k}(p,C)b_C^1)=\ker b_B^1=\hh^0(B,C). \qedhere\]
\end{proof}

\begin{remark} In the case of cluster-tilted algebras, it was proved in \cite{AR,ABIS,ARS} that the first Hochschild projection morphism $\varphi^1:\HH^1(B)\rightarrow \HH^1(C)$ is surjective. It is therefore natural to ask whether the Hochschild projection morphism $\varphi^*:\HH^{*}(B)\rightarrow \HH^*(C)$ is surjective or not. This is not the case in general, even for trivial extensions, as the following example shows, but we shall see some important cases where it is in Section \ref{sec:trivial extensions}.
\end{remark}

\begin{example}\label{not surjective} Let $B$ be the algebra given by the quiver 
\[ \xymatrix@R=10pt@C=10pt{1\ar[rr]^\alpha\ar[dr]_\beta&&2\ar@(dr,ur)[]_\varepsilon\\&3\ar[ur]_\gamma} \] bound by $\varepsilon^2=0$ and $\alpha\varepsilon=\beta\gamma\varepsilon.$ Let $E$ be the ideal in $B$ generated by $\varepsilon$ and let $C$ be the path algebra of the quiver \[ \xymatrix@R=10pt@C=10pt{1\ar[rr]^\alpha\ar[dr]_\beta&&2\\&3\ar[ur]_\gamma} \] Then $B$ is the trivial extension of $C$ by $E$, and the maps $p:B\rightarrow C$ and $q:C\rightarrow B$ are the natural maps. We assume that the characteristic of the field $k$ is not $2.$

Since $C$ is hereditary, $\dim_k\HH^1(C)=2$ by \cite[Proposition 1.6]{H}. Moreover, $B$ can be obtained by taking the one-point extension of the one-point extension of the algebra $k[x]/(x^2)$, and therefore it follows using Theorem~\ref{thm 1.1} that $\dim_k\HH^1(B)=3.$

In order to describe $\varphi^1$, we need explicit bases of $\HH^1(C)$ and $\HH^1(B)$. 

Any normalised derivation $d:C\rightarrow C$ is entirely determined by 
\[ 
\begin{cases}
d(\alpha)=\lambda_1\alpha+\lambda_2\beta\gamma\\d(\beta)=\lambda_3\beta\\d(\gamma)=\lambda_4\gamma.
\end{cases}
 \] Moreover, the inner derivations $d_1$ and $d_2$ associated to the vertices $e_1$ and $e_2$ are determined by 
\begin{align*}
\begin{cases}
d_1(\alpha)&=-\alpha \\d_1(\beta)&=-\beta \\d_1(\gamma)&=0
\end{cases}
&\qquad \text{and}\qquad
\begin{cases}
d_2(\alpha)&=\alpha\\d_2(\beta)&=0\\d_2(\gamma)&=\gamma.
\end{cases}
\end{align*}
 It then follows that modulo inner derivations, the normalised derivation $d$ above is equivalent to the normalised derivation $d'$ determined by $d'(\alpha)=(\lambda_1-\lambda_3-\lambda_4)\alpha+\lambda_2\beta\gamma,$ $d'(\beta)=0$ and $d'(\gamma)=0.$ Since $\dim_k\HH^1(C)=2$, it follows that  a basis of $\HH^1(C)$ is given by the normalised derivations $u_1$ and $u_2$ determined by $u_1(\alpha)=\alpha$, $u_2(\alpha)=\beta\gamma$ and $u_i(\beta)=0=u_i(\gamma)$ for $i=1,2$. 

Any normalised derivation $d:B\rightarrow B$ is entirely determined by 
\[ 
\begin{cases}
d(\alpha)=\lambda_1\alpha+\lambda_2\beta\gamma+\mu_1\alpha\varepsilon\\d(\beta)=\lambda_3\beta\\d(\gamma)=\lambda_4\gamma+\mu_2\gamma\varepsilon\\d(\varepsilon)=\mu_0e_2+\mu_3\varepsilon.
\end{cases}
 \] Moreover, since $0=d(\varepsilon^2)=2\mu_0\varepsilon$ we must have $\mu_0=0$, and $d(\alpha\varepsilon-\beta\gamma\varepsilon)=0$ implies that $\lambda_1+\lambda_2-\lambda_3-\lambda_4=0$. It then follows as above that modulo inner derivations (associated to $e_1$, $e_2$ and $\varepsilon$), $d$ is equivalent  to the normalised derivation $d'$ defined by $d'(\alpha)=(\lambda_1-\lambda_3-\lambda_4)\alpha+\lambda_2\beta\gamma+(\mu_1-\mu_2)\alpha\varepsilon=\lambda_2(\alpha-\beta\gamma)+(\mu_1-\mu_2)\alpha\varepsilon$ and $d'(\varepsilon)=\mu_3\varepsilon.$ Since $\dim_k\HH^1(B)=3$, it follows that a basis of $\HH^1(B)$ is given by the normalised derivations $v_1$, $v_2$ and $v_3$ determined by $v_1(\varepsilon)=\varepsilon$, $v_2(\alpha)=\alpha-\beta\gamma$, $v_3(\alpha)=\alpha\varepsilon$ and all other arrows are sent to $0$. 

Then $v_1q=0=v_3q$ hence $\varphi^1([v_1])=0=\varphi^1([v_3])$ and therefore $\dim_k \im\varphi^1\ppq 1<\dim_k \HH^1(C)$ so that $\varphi^1$ cannot be surjective. Note that $\varphi^1([v_2])=[u_1-u_2]\neq 0.$

In particular, $\varphi^*$ is not surjective.
\end{example}

\section{Trivial extensions}\label{sec:trivial extensions}

In this section, we consider the case where $B=C\ltimes E$ is a trivial extension, that is, $E^2=0.$ Our objective is to give a necessary and sufficient condition for the Hochschild projection morphisms $\varphi^n:\HH^n(B)\rightarrow \HH^n(C)$ to be surjective. We then prove that this condition holds for the two important cases of trivial extension algebras, namely those where $E=\, _CDC_C$ is the minimal injective cogenerator bimodule, and those where $E=\,_CC_C$ is the regular bimodule. Our main concern however is with the relation bimodule $E=\Ext^2_C(DC,C)$ of a triangular algebra $C$ of global dimension two; in this case, $B$ is the relation extension of $C$, see \cite{ABS1}. We give a necessary and sufficient condition for the first Hochschild projection morphism $\varphi^1$ to be surjective, valid even in the more general case where $C$ is not assumed to be triangular of global dimension two and $E=E_m=\Ext^m_C(DC,C)$. Such trivial extensions $B=C\ltimes E_m$ are related, under some conditions, to $(m-1)$-cluster-tilted algebras, as was mentioned in the introduction.

We start with some background from \cite{CMRS}.

Let $B=C\ltimes E$ be the split extension  of $C$ by $E$. As $C$-$C$-bimodules, we have $B=E\oplus C$ and therefore $B^{\otimes n}=\bigoplus_{s+r=n}E^{s,r}$ where  $E^{s,r}$ is the subspace of $B^{\otimes(s+r)}$ generated by the tensors $x_1\otimes\cdots\otimes x_{s+r}$ with exactly $s$ of the $x_i$ in $E$ and $r$ of the $x_i$ in $C$. 

Consequently, if $X$ is any $B$-$B$-bimodule, the Hochschild complex $\Hom_k(B^{\otimes *},X)$ organises into a double complex $\Hom_k(E^{*,*},X)$, and the Hochschild differential $b_X^{n+1}$ decomposes into  horizontal differentials $d_h^{s+1,r}:\Hom_k(E^{s,r},X)\rightarrow \Hom_k(E^{s+1,r},X)$ and  vertical differentials $d_v^{s,r+1}:\Hom_k(E^{s,r},X)\rightarrow\Hom_k(E^{s,r+1},X)$ where $s$  and $r$ are non-negative integers such that $s+r=n$, that is, $b_X^{n+1}=\sum_{s+r=n}(d_h^{s+1,r}+d_v^{s,r+1})$. This can be illustrated in the following diagram.
{\tiny\[ \xymatrix{\Hom_k(E^{0,n+1},X)\ar@{.}[rd]|{\bigoplus}\\
\Hom_k(E^{0,n},X)\ar[u]_{d_v^{0,n+1}}\ar[r]_{d_h^{1,n}}\ar@{.}[rd]|{\bigoplus}&\Hom_k(E^{1,n},X)\ar@{--}[rrdd]|{\bigoplus}\\
&\Hom_k(E^{1,n-1},X)\ar[u]_{d_v^{1,n}}\ar@{--}[rrdd]|{\bigoplus}\\
&&&\Hom_k(E^{s,r+1},X)\ar@{.}[rd]|{\bigoplus}\\
&&&\Hom_k(E^{s,r},X)\ar[u]_{d_v^{s,r+1}}\ar[r]_{d_h^{s+1,r}}\ar@{--}[rd]|{\bigoplus}&\Hom_k(E^{s+1,r},X)\ar@{--}[rd]|{\bigoplus}\\
&&&&\Hom_k(E^{n,0},X)\ar[r]_{d_h^{n+1,0}}&\Hom_k(E^{n+1,0},X)} \]}

We shall denote the $s$\tup{th} column in this double complex (with vertical differential) by $\calc^s(X).$   

The next proposition was proved in \cite{CMRS}.

\begin{prop}[\cite{CMRS}]\label{prop:results CMRS} The following results hold for the split extension $B$ of $C$ by $E$:
\begin{enumerate}[(a)]
\item\label{cmrs:col0} $\hh^n(\calc^0(X))=\Ext_{C\da C}^n(C,X)=\hh^n(C,X)$. 
\item\label{cmrs:col1} $\hh^n(\calc^1(X))=\Ext_{C\da C}^n(E,X)$.
\item\label{cmrs:colm} $\hh^0(\calc^s(X))=\Hom_{C\da C}(E^{\otimes_C s},X)$.
\end{enumerate}
If moreover $B=C\ltimes E$ is the trivial extension of $C$ by $E$ and if $X$ satisfies $XE=0=EX$, then:
\begin{enumerate}[(a),resume]
\item\label{cmrs:horizontal diff} the horizontal differential $d_h$ is zero;
\item\label{cmrs:HHn decomposition} $\hh^n(B,X)=\bigoplus_{s+r=n}\hh^r(\calc^s(X))$;
\item\label{cmrs:connecting morphism} the connecting morphism $\delta_B^n:\hh^n(B,C)\rightarrow\hh^{n+1}(B,E)$ decomposes into $\delta_B^n=\bigoplus_{s+r=n} \delta^{s,r}$ where $\delta^{s,r}:\hh^r(\calc^s(C))\rightarrow\hh^r(\calc^{s+1}(E))$. 

Moreover, $\delta^{s,r}([\zeta])=[\id_E\smile\zeta]+(-1)^{s+r+1}[\zeta\smile\id_E]$ where, if $f:B^{\otimes s}\rightarrow C$ and $g:B^{\otimes r}\rightarrow C$, $f\smile g:B^{\otimes(s+r)}\xrightarrow{f\otimes g} C\otimes C\rightarrow C\otimes_B C$. 
\end{enumerate}
\end{prop}

\begin{proof}
\ref{cmrs:col0} is in \cite[p. 24]{CMRS}, \ref{cmrs:col1} is in   \cite[Theorem 2.2]{CMRS}, \ref{cmrs:colm} is in \cite[Remark 2.3]{CMRS}, \ref{cmrs:horizontal diff} and  \ref{cmrs:HHn decomposition} follow from \cite[Theorem 3.1]{CMRS} and \ref{cmrs:connecting morphism} from \cite[Proposition 3.6 and Theorem 4.1]{CMRS}.
\end{proof}

\begin{cor} Let $B$ be the trivial extension of $C$ by $E.$ We have a decomposition \[\hh^n(B,C)\cong \bigoplus_{s+r=n}\hh^r(\calc^s(C))=\HH^n(C)\oplus\bigoplus_{s+r=n, \ s>0}\hh^r(\calc^s(C)).\]
\end{cor}

\begin{proof} Apply  Proposition~\ref{prop:results CMRS}\ref{cmrs:HHn decomposition} with $X=q(C)\cong C$.
\end{proof}

We now give the connection with the Hochschild projection morphisms.

\begin{lem}\label{lemma:phin reinterpreted} \sloppy Let $\psi_n$ be the isomorphism $\hh^n(B,C)\rightarrow\bigoplus_{s+r=n}\hh^r(\calc^s(C))$ and let $\pi_{0,n}:\bigoplus_{s+r=n}\hh^r(\calc^s(C))\rightarrow \HH^n(C)$ be the natural projection. 

Then $\pi_{0,n}\psi_n=\nu_n$ is the retraction in Lemma~\ref{cor:cohom C embeds B}, and if $\HH^n(p):\HH^n(B)\rightarrow\hh^n(B,C)$ is the morphism induced on the Hochschild cohomology by $p$, then $\pi_{0,n}\psi_n\HH^n(p)=\varphi^n$ is the Hochschild projection morphism.
\end{lem}

\begin{proof}
The decomposition $\hh^n(B,C)\cong \bigoplus_{s+r=n}\hh^r(\calc^s(C))$ arises from the isomorphism $(q,i):C\oplus E\rightarrow B$ of $k$-vector spaces which induces an isomorphism \[\Hom_k(B^{\otimes n},C)\xrightarrow\cong  \Hom_k(C^{\otimes n},C)\oplus\bigoplus_{s+r=n, \ s>0}\Hom_k(E^{s,r},C)\] on the level of cochains, therefore the projection onto the first component is  given by $[f]\mapsto [fq^{\otimes n}]=\nu_n([f])$. 

Moreover, if $f\in\Hom_k(B^{\otimes n},B)$ then $\nu_n\HH^n(p)([f])=\nu_n([pf])=[pfq^{\otimes n}]=\varphi^n([f]),$ which concludes the proof.
\end{proof}

\begin{remark}\label{rk: description projection on cale E}\sloppy
In the same way, we can prove that if $\pi_{n,n}:\bigoplus_{s+r=n}\hh^r(\calc^s(C))\rightarrow H^0(\calc^n(C))=\Hom_{C\da C}(E^{\otimes_Cn},C)$ is the projection onto the last component, then $\pi_{n,n}\psi_n:\hh^n(B,C)\rightarrow\Hom_{C\da C}(E^{\otimes_Cn},C)$  is given by $[f]\mapsto fi^{\otimes n}$, where $i:E\rightarrow B$ is the embedding, and $\pi_{n,n}\psi_n\HH^n(p):\HH^n(B)\rightarrow \Hom_{C\da C}(E^{\otimes_Cn},C)$ is given by $[f]\mapsto pfi^{\otimes n}$.
\end{remark}

We shall now study the surjectivity of $\varphi^*.$ We keep the notation above.

\begin{prop}\label{prop: SC for surjectivity}
 For any $n\pgq0,$ the $n$\tup{th} Hochschild projection morphism $\varphi^n:\HH^n(B)\rightarrow \HH^n(C)$ is surjective if, and only if, $\delta^{0,n}=0$.
\end{prop}

\begin{proof} The long exact cohomology sequence \eqref{eq:long cohomology sequence} on page \pageref{eq:long cohomology sequence} combined with the decomposition in Lemma~\ref{lemma:phin reinterpreted} yield an exact sequence 
\[ \HH^n(B)\xrightarrow{\psi_n\HH^n(p)}\HH^n(C)\oplus\bigoplus_{s+r=n,\ s>0}\hh^r(\calc^s(C)) \xrightarrow{\delta_B^n\psi_n^{-1}} \hh^{n+1}(B,E).\]
Therefore $\varphi^n=\pi_{0,n}\psi_n\HH^n(p)$ is surjective if, and only if, $\HH^n(C)\subset \im\psi_n\HH^n(p)=\ker\delta_B^n\psi_n^{-1} $ which is equivalent to $\delta^{0,n}=(\delta_B^n\psi_n^{-1})_{|\HH^n(C)}=0.$
\end{proof}

We shall say that a $C$-$C$-bimodule $E$ is \emph{symmetric over} $Z(C)$ if it is symmetric when viewed as a bimodule over the centre $Z(C)$ of $C,$ that is, for any $z\in Z(C)$ and any $x\in E$ we have $zx=xz.$
For instance, if $C$ is a triangular algebra (that is, $C$ has an acyclic quiver), then $Z(C)\cong k$ and any $C$-$C$-bimodule is symmetric over $Z(C).$

\renewcommand{\theequation}{C\arabic{equation}}
\begin{cor}\label{cor:technical condition} 
\begin{enumerate}[(a)]
\item The Hochschild projection morphism $\varphi^0$ is surjective if, and only if, $E$ is symmetric over $Z(C)$.
\item  For $n\pgq 1$, the Hochschild projection morphism $\varphi^n$ is surjective if, and only if, for any Hochschild cocycle $\zeta\in\Hom_k(C^{\otimes n},C)$, there exists a morphism $\alpha\in\Hom_k(E^{1,n-1},E)$ that satisfies the following three conditions for all $\theta\in E$ and all $\ul{c}=c_1\otimes\cdots\otimes c_n\in C^{\otimes n}$:
\begin{align}
\begin{split}\label{eq:c1}\theta\cdot \zeta(\ul{c})&=-\alpha(\theta\cdot c_1\otimes c_2\otimes \cdots \otimes c_n)+\sum_{i=1}^{n-1}(-1)^{i+1}\alpha(\theta\otimes c_1\otimes \cdots\otimes c_ic_{i+1}\otimes \cdots \otimes c_n)\\&\qquad+(-1)^{n+1}\alpha(\theta\otimes c_1\otimes \cdots \otimes c_{n-1})\cdot c_n
\end{split}\\
\begin{split}\label{eq:c2}(-1)^{n+1}\zeta(\ul{c})\cdot\theta &=c_1\cdot\alpha(c_2\otimes c_3\otimes \cdots \otimes c_n\otimes \theta)+\sum_{i=1}^{n-1}(-1)^{i}\alpha(c_1\otimes \cdots\otimes c_ic_{i+1}\otimes \cdots \otimes c_n\otimes \theta)\\&\qquad+(-1)^{n}\alpha(c_1\otimes \cdots \otimes c_{n-1}\otimes c_n\cdot \theta)
\end{split}\\
\begin{split}\label{eq:c3}0&=d_v\alpha(c_1\otimes\cdots\otimes c_i\otimes \theta\otimes c_{i+1}\otimes\cdots\otimes c_n)\quad\text{for all } i=1,2,\ldots,n-1.
\end{split}
\end{align}
\end{enumerate}
\end{cor}

\begin{proof} 
\begin{enumerate}[(a)]
\item We have $\HH^0(B)=Z(B)$, $\HH^0(C)=Z(C)$, and $\varphi^0:Z(B)\rightarrow Z(C)$ is the restriction of the projection $p:B\rightarrow C$ to $Z(B)$. 
\begin{enumerate}[$\triangleright$]
\item Assume that $E$ is symmetric over $Z(C)$. If $z\in Z(C)$, then $z=p(z,0)$. Moreover, for any $(c,x)\in B=C\ltimes E$, we have 
\[(z,0)(c,x)-(c,x)(z,0)=(zc,zx)-(cz,xz)=(zc-cz,zx-xz)=0,\] therefore $(z,0)\in Z(B)$ and $z=\varphi^0(z,0)$. Hence $\varphi^0$ is surjective.
\item Assume that $\varphi^0$ is surjective. Take $z\in Z(C)$ and $x\in E.$ There exists $(c,y)\in Z(B)$ such that $z=\varphi^0(c,y)=p(c,y)=c$. Therefore $(z,y)\in Z(B)$ and we have 
\[ 0=(z,y)(0,x)-(0,x)(z,y)=(0,zx-xz) \] so that $zx=xz$. 

This is true for any $z\in Z(C)$ and any $x\in E$, therefore $E$ is symmetric over $Z(C)$.
\end{enumerate}
Note that this can also be deduced from Proposition~\ref{prop: SC for surjectivity} as follows. 
If $n=0,$ then $\delta^{0,0}=\delta^0_B:\hh^0(B,C)\rightarrow\hh^0(B,E)$. Moreover, by 
 \cite[Proposition 3.3]{CMRS}, the connecting morphism $\delta^0_B$ vanishes if, and only if, $E$ is symmetric over $Z(C)$. Claim (a) now follows from Proposition~\ref{prop: SC for surjectivity}.

\item Now assume that $n\pgq 1.$ The map $\delta^{0,n}$ vanishes if, and only if, for any cocycle $\zeta\in\Hom_k(C^{\otimes n},C)$, the map $f_\zeta=\id_E\smile\zeta+(-1)^{n+1}\zeta\smile\id_E$ representing  $\delta^{0,n}([\zeta])$ is a coboundary. Since the horizontal differential $d_h$ vanishes, we must prove that $f_\zeta$ is in the image of the vertical differential. Since $f_\zeta\in\Hom_k(E^{1,n},E)$, we must find $\alpha\in\Hom_k(E^{1,n-1},E)$ such that $d_v(\alpha)=f_\zeta$. 

The formula $f_\zeta=\id_E\smile\zeta+(-1)^{n+1}\zeta\smile\id_E$ can be expressed as follows:
\begin{align*}
(f_\zeta)_{|E\otimes C^{\otimes n}}&=\id_E\smile \zeta\\
(f_\zeta)_{|C^{\otimes n}\otimes E}&=(-1)^{n+1}\zeta\smile\id_E
\end{align*} and $f_\zeta$ restricted to the other components in $E^{1,n}$ vanishes.
This translates as: for any $\theta\in E$ and any $\ul{c}=c_1\otimes \cdots \otimes c_n\in C^{\otimes n}$, 
\begin{align*}
f_\zeta(\theta\otimes \ul{c})&=\theta\cdot\zeta(\ul{c})\\
f_\zeta(\ul{c}\otimes \theta)&=\zeta(\ul{c})\cdot\theta\\
f_\zeta(c_1\otimes \cdots&\otimes c_i\otimes \theta\otimes c_{i+1}\otimes \cdots\otimes c_n)=0\text{ for $1\ppq i\ppq n-1$.}
\end{align*}
The vertical differential $d_v$ is just the Hochschild differential (co-)restricted to the spaces in a given column $\calc^s(X)$, therefore $d_v(\alpha)$ evaluated at an element in one of the components of $E^{1,n}$ is the right hand side of the appropriate equation  \eqref{eq:c1}, \eqref{eq:c2} or \eqref{eq:c3}. It follows that  $d_v(\alpha)=f_\zeta$ if and only if  $\alpha$ satisfies Conditions \eqref{eq:c1}, \eqref{eq:c2} and \eqref{eq:c3}.
\qedhere
\end{enumerate}

\end{proof}

Recall that the vector space $DC$ is a $C$-$C$-bimodule for the usual actions, that is, for any  $f\in DC$, $c\in C$ and $x\in C$, \[(f\cdot c)(x)=f(c x) \text{ and } (c\cdot f)(x)=f(xc).\]

\begin{cor} Assume that $B=C\ltimes DC$. Then for all $n\pgq 0$, the Hochschild projection morphism $\varphi^n:\HH^n(B)\rightarrow \HH^n(C)$ is surjective, so that $\varphi^*:\HH^*(B)\rightarrow \HH^*(C)$ is surjective.
\end{cor}

\begin{proof}
It was proved in \cite[Proposition 5.9]{CMRS} that $\delta^{0,n}=0$ for any $n\pgq 1$ in this case. If $\zeta$ is any Hochschild cocycle in $\Hom_k(C^{\otimes n},C)$, the morphism $\alpha\in\Hom_k(E^{1,n-1},E)$ satisfying relations \eqref{eq:c1}, \eqref{eq:c2} and \eqref{eq:c3} was given explicitly on each component $C^{\otimes p}\otimes DC\otimes C^{\otimes q}$ by $\alpha(\ul{x}\otimes \theta\otimes \ul{y})(c)=(-1)^{n(p+1)+1}\theta(\zeta(\ul{y}\otimes c\otimes \ul{x}))$ for all $\ul{x}\in C^{\otimes p}$, $\ul{y}\in C^{\otimes q},$ $\theta\in E=DC$ and $c\in C.$ It follows from Proposition~\ref{prop: SC for surjectivity} that for all $n\pgq 1$ the Hochschild projection morphism $\varphi^n$ is surjective.

Finally, it is easy to check that the $C$-$C$-bimodule $DC$ is symmetric over $Z(C)$, therefore $\varphi^0$ is surjective. 
\end{proof}

\begin{cor} Assume that $B=C\ltimes C$. Then for all $n\pgq 0$, the Hochschild projection morphism $\varphi^n:\HH^n(B)\rightarrow \HH^n(C)$ is surjective.
\end{cor}

\begin{proof} Clearly, the bimodule $C$ is symmetric over $Z(C)$, therefore $\varphi^0$ is surjective. 

For $n\pgq 1,$ 
 it is easy to check that  if $\zeta$ is any Hochschild cocycle in $\Hom_k(C^{\otimes n},C)$, the morphism $\alpha=-\zeta\in\Hom_k(E^{1,n-1},E)=\Hom_k(C^{\otimes n},C)$ satisfies Conditions \eqref{eq:c1}, \eqref{eq:c2} and \eqref{eq:c3}. Let us for instance check \eqref{eq:c1} and leave the other conditions  to the reader. Since $b_B^{n+1}(\zeta)=0$, we have, for any $\theta\in E=C$ and any $\ul{c}=c_1\otimes \cdots\otimes c_n\in C^{\otimes n}$, 
\begin{align*}
0&=b_B^{n+1}(\zeta)(\theta\otimes \ul{c})=\theta\zeta(\ul{c})-\zeta(\theta c_1\otimes c_2\otimes\cdots\otimes c_n)\\
&\qquad-\sum_{i=1}^{n-1}(-1)^{i}\zeta(\theta\otimes c_1\otimes \cdots\otimes c_{i}c_{i+1}\otimes \cdots\otimes c_n)-(-1)^n\zeta(\theta\otimes c_1\otimes \cdots\otimes c_{n-1})c_n\\
&=\theta\zeta(\ul{c})+\alpha(\theta c_1\otimes c_2\otimes\cdots\otimes c_n)\\
&\qquad+\sum_{i=1}^{n-1}(-1)^{i}\alpha(\theta\otimes c_1\otimes \cdots\otimes c_{i}c_{i+1}\otimes \cdots\otimes c_n)+(-1)^n\alpha(\theta\otimes c_1\otimes \cdots\otimes c_{n-1})c_n
\end{align*} and \eqref{eq:c1} follows.
\end{proof}

We shall now consider the case where $E=E_m=\Ext^m_C(DC,C)$ where $m$ is a non-negative integer and $n=1$.
In this situation, $\zeta\in\Hom_k(C,C)$ is a derivation (since $n=1$). 

The space $E_m=\Ext^m_C(DC,C)$ is  naturally a $C$-$C$-bimodule, where the left action is on the target $C$ of $\theta$ and the right action is given by the left action on the component $DC$ of the argument of $\theta$. More precisely,  if $ c \in C$ and $\theta\in\Hom_k(DC\otimes C^{\otimes m},C)$ is a cocycle representing an element in $E_m$, then for all $\ul{a}\in C^{\otimes m}$ and all $f\in DC$ we have 

\setcounter{equation}{0}
\renewcommand{\theequation}{\thesection.\arabic{equation}}

\begin{equation}\label{eq:actions E}
\begin{split}
( c \cdot\theta)(f\otimes \ul{a})&= c\,  \theta  (f\otimes \ul{a}) \\
(\theta\cdot  c )(f\otimes \ul{a})&=\theta(c\cdot f \otimes \ul{a}).
\end{split}
\end{equation}

\begin{defn}
Given a derivation $\zeta\in\Hom_k(C,C)$ and an integer $m\pgq 0$, define $\alpha_m:\Hom_k(DC\otimes C^{\otimes m},C)\rightarrow \Hom_k(DC\otimes C^{\otimes m},C)$ by 
\[ \alpha_m(\theta)(f\otimes \ul{a})=\sum_{j=1}^s \theta(f\otimes a_1\otimes \cdots\otimes \zeta(a_j)\otimes \cdots\otimes a_m)-\theta(f\circ \zeta\otimes \ul{a})-\zeta(\theta(f\otimes \ul{a}))\] for all $\theta\in \Hom_k(DC\otimes C^{\otimes m},C)$, for all $\ul{a}=a_1\otimes\cdots\otimes a_m \in C^{\otimes m}$ and for all $f\in DC.$
\end{defn}

We shall prove in Corollary~\ref{cor:cocycle and coboundary} and Lemma~\ref{lemma:alphas satisfy equations} that $\alpha_m$ defines a morphism in $\Hom_k(E_m^{1,n-1},E_m)=\End_kE_m$ which satisfies Conditions \eqref{eq:c1} and \eqref{eq:c2},  Condition \eqref{eq:c3} being empty for $n=1$.

\begin{remark}
For any derivation $\zeta\in\Hom_k(C,C)$,  if $f\in DC$ and $ c \in C$, we have the following relations:
\begin{equation}\label{eq:actions compo}
\begin{split}
 c \cdot(f\circ \zeta)&=( c \cdot f)\circ \zeta+\zeta( c )\cdot f\\
(f\circ \zeta)\cdot  c &=(f\cdot c )\circ \zeta+f\cdot\zeta( c ).
\end{split} 
\end{equation} Indeed, since $\zeta$ is a derivation, for any $a\in C$ we have 
\begin{align*}
( c \cdot(f\circ \zeta))(a)=f(\zeta(a c ))&=f(\zeta(a) c +a\zeta( c ))=f(\zeta(a) c )+f(a\zeta( c ))\\
&=( c \cdot f)(\zeta(a))+(\zeta( c )\cdot f)(a)=\left(( c \cdot f)\circ \zeta+\zeta( c )\cdot f\right)(a)
\end{align*} and the proof of the second relation is similar.
\end{remark}

In the sequel,  denote by $\partial $ the differential in the complex $\Hom_{k}(DC\otimes C^{\otimes *},C)$ whose cohomology is $\Ext_C^*(DC,C)$.

\begin{lem}\label{lemma:cocycle and coboundary} For every integer $m\pgq 0$ and every $\theta\in \Hom_k(DC\otimes C^{\otimes m},C)$, we have 
\[\partial (\alpha_m(\theta))=\alpha_{m+1}(\partial \theta), \] that is, there is a commutative diagram 
\[ \xymatrix{\Hom_k(DC\otimes C^{\otimes m},C) \ar[r]^{\alpha_m}\ar[d]_{\Hom_k(\partial ,C)}  & \Hom_k(DC\otimes C^{\otimes m},C)\ar[d]^{\Hom_k(\partial ,C)} \\
\Hom_k(DC\otimes C^{\otimes (m+1)},C) \ar[r]^{\alpha_{m+1}} &\Hom_k(DC\otimes C^{\otimes (m+1)},C)}. \]
\end{lem}

\begin{proof}
Take $\ul{a}=a_0\otimes \cdots \otimes a_m\in C^{\otimes {m+1}}$ and $f\in DC.$ Then:
\begin{align*}
\partial (\alpha_m(\theta))(f\otimes \ul{a})&=\alpha_m(\theta)(f\cdot a_0\otimes a_1\otimes \cdots \otimes a_m)+\sum_{i=1}^m(-1)^i\alpha_m(\theta)(f\otimes a_0\otimes\cdots\otimes a_{i-1}a_i\otimes \cdots \otimes a_m)\\
&\qquad +(-1)^{m+1}\alpha_m(\theta)(f\otimes a_0\otimes \cdots\otimes a_{m-1})a_m\\
&=\sum_{j=1}^m\theta(f\cdot a_0\otimes a_1\otimes \cdots\otimes \zeta(a_j)\otimes\cdots \otimes a_m)-\theta((f\cdot a_0)\circ \zeta\otimes a_1\otimes \cdots \otimes a_m)
\\&\qquad\qquad-\zeta(\theta(f\cdot a_0\otimes a_1\otimes \cdots \otimes a_m))\\
&\qquad+\sum_{i=1}^m(-1)^i\Big[\sum_{j=0}^{i-2}\theta(f\otimes a_0\otimes\cdots\otimes \zeta(a_j)\otimes \cdots \otimes a_{i-1}a_i\otimes \cdots\otimes a_m)\\
&\hphantom{\sum_{i=1}^m(-1)^i\qquad\sum_{j=0}^{s-2}\theta}+\theta(f\otimes a_0\otimes \cdots \otimes \zeta(a_{i-1}a_i)\otimes \cdots \otimes a_m)\\
&\hphantom{\sum_{i=1}^m(-1)^i\qquad\sum_{j=0}^{s-2}\theta}+\sum_{j={i+1}}^m\theta(f\otimes a_0\otimes\cdots\otimes a_{i-1}a_i\otimes \cdots \otimes \zeta(a_j)\otimes \cdots \otimes a_m)\\
&\hphantom{\sum_{i=1}^m(-1)^i\qquad\sum_{j=0}^{s-2}\theta}-\theta(f\circ \zeta\otimes a_0\otimes \cdots \otimes a_{i-1}a_i\otimes \cdots \otimes a_m)\\
&\hphantom{\sum_{i=1}^m(-1)^i\qquad\sum_{j=0}^{s-2}\theta}-\zeta(\theta(f\otimes a_0\otimes \cdots\otimes a_{i-1}a_i\otimes \cdots \otimes a_m))\Big]\\
&\qquad+(-1)^{m+1}\Big[\sum_{j=0}^{m-1} \theta(f\otimes a_0\otimes \cdots\otimes \zeta(a_j)\otimes \cdots\otimes a_{m-1})a_m\\&\hphantom{\qquad+(-1)^{m+1}\sum_{j=0}^{s-1} }-\theta(f\circ \zeta\otimes a_0\otimes\cdots \otimes a_{m-1})a_m-\zeta(\theta(f\otimes a_0\otimes\cdots \otimes a_{m-1}))a_m\Big]
\end{align*}  by definition of $\alpha_m.$ We now exchange the sums over $i$ and $j$,  use relation \eqref{eq:actions compo} in the second line, and the fact that $\zeta$ is a derivation in the fifth and last lines below, to get
\begin{align*}
\partial (\alpha_m(\theta))(f\otimes \ul{a})&=\sum_{j=1}^m\theta(f\cdot a_0\otimes a_1\otimes \cdots\otimes \zeta(a_j)\otimes\cdots \otimes a_m)
\\
&\quad\stackrel{\eqref{eq:actions compo}}{+}\theta\left(\big[-(f\circ \zeta)\cdot a_0+f\cdot \zeta(a_0)\big]\otimes a_1\otimes \cdots \otimes a_m\right)
\\
&\quad -\zeta(\theta(f\cdot a_0\otimes a_1\otimes \cdots \otimes a_m))
\\
&\quad +\sum_{j=0}^{m-2}\sum_{i=j+2}^m(-1)^i\theta(f\otimes a_0\otimes \cdots \otimes \zeta(a_j)\otimes \cdots \otimes a_{i-1}a_i\otimes \cdots \otimes a_m)
\\
&\quad +\sum_{i=1}^m(-1)^i\theta(f\otimes a_0\otimes \cdots \otimes [a_{i-1}\zeta(a_i)+\zeta(a_{i-1})a_i]\otimes \cdots a_m)
\\
&\quad +\sum_{j=1}^m\sum_{i=1}^{j-1}(-1)^i\theta(f\otimes a_0\otimes\cdots \otimes a_{i-1}a_i\otimes \cdots \otimes \zeta(a_j)\otimes \cdots \otimes a_m)
\\
&\quad +\sum_{i=1}^m(-1)^{i+1}\theta(f\circ \zeta\otimes a_0\otimes \cdots \otimes a_{i-1}a_i\otimes \cdots \otimes a_m)
\\
&\quad +\sum_{i=1}^m(-1)^{i+1}\zeta(\theta(f\otimes a_0\otimes \cdots \otimes a_{i-1}a_i\otimes \cdots \otimes a_m))\\
&\quad+(-1)^{m+1}\Big[\sum_{j=0}^{m-1} \theta(f\otimes a_0\otimes \cdots\otimes \zeta(a_j)\otimes \cdots\otimes a_{m-1})a_m\\&\hphantom{\qquad+(-1)^{m+1}\sum_{j=0}^{s-1} }-\theta(f\circ \zeta\otimes a_0\otimes\cdots \otimes a_{m-1})a_m\\&\hphantom{\qquad+(-1)^{m+1}\sum_{j=0}^{s-1} }-\zeta(\theta(f\otimes a_0\otimes \cdots\otimes a_{m-1})a_m)+\theta(f\otimes a_0\otimes \cdots \otimes a_{m-1})\zeta(a_m)\Big]
\end{align*}   The fifth line can  be rewritten as 
\[ \sum_{i=1}^m(-1)^i\theta(f\otimes a_0\otimes \cdots \otimes a_{i-1}\zeta(a_i)\otimes \cdots a_m)+\sum_{i=0}^{m-1}(-1)^{i+1}\theta(f\otimes a_0\otimes \cdots \otimes \zeta(a_{i})a_{i+1}\otimes \cdots a_m). \]
We now rewrite the expression we have obtained for $\partial (\alpha_m(\theta))(f\otimes \ul{a}),$ keeping the terms in the same order but isolating some terms for $j=0$ or $j=m$ and relabelling the summation index from $i$ to $j$ in some cases.
\begin{align*}
\partial (\alpha_m(\theta))(f\otimes \ul{a})&=\sum_{j=1}^{m-1}\theta(f\cdot a_0\otimes a_1\otimes \cdots \otimes \zeta(a_j)\otimes \cdots \otimes a_{m})+\theta(f\cdot a_0\otimes a_1\otimes \cdots \otimes a_{m-1}\otimes \zeta(a_m))\\
&\qquad-\theta((f\circ \zeta)\cdot a_0\otimes a_1\otimes \cdots \otimes a_{m})+\theta(f\cdot\zeta(a_0)\otimes a_1\otimes \cdots\otimes a_m)\\
&\qquad-\zeta(\theta(f\cdot a_0\otimes a_1\otimes\cdots\otimes a_m))\\
&\qquad+\sum_{i=2}^m(-1)^i\theta(f\otimes \zeta(a_0)\otimes\cdots\otimes a_{i-1}a_i\otimes\cdots\otimes a_m)\\
&\qquad+\sum_{j=1}^{m-2}\sum_{i={j+2}}^m(-1)^i\theta(f\otimes a_0\otimes \cdots\otimes a_{j-1}\otimes \zeta(a_j)\otimes \cdots \otimes a_{i-1}a_i\otimes \cdots \otimes a_m)\\
&\qquad+\sum_{j=1}^{m-1}(-1)^j\theta(f\otimes a_0\otimes \cdots \otimes a_{j-2}\otimes a_{j-1}\zeta(a_j)\otimes \cdots \otimes a_{m})\\
&\qquad\qquad+(-1)^m\theta(f\otimes a_0\otimes \cdots \otimes a_{m-1}\zeta(a_m))\\
&\qquad-\theta(f\otimes \zeta(a_0)a_1\otimes a_2\otimes\cdots\otimes a_m)\\
&\qquad\qquad+\sum_{j=1}^{m-1}
(-1)^{j+1}\theta(f\otimes a_0\otimes \cdots\otimes a_{j-1}\otimes \zeta(a_j)a_{j+1}\otimes \cdots \otimes a_m)\\
&\qquad+\sum_{j=1}^{m-1}\sum_{i=1}^{j-1}(-1)^i\theta(f\otimes a_0\otimes \cdots \otimes a_{i-1}a_i\otimes \cdots \otimes \zeta(a_j)\otimes \cdots \otimes a_m)\\
&\qquad\qquad+\sum_{i=1}^{m-1}(-1)^{i}\theta(f\otimes a_0\otimes \cdots \otimes a_{i-1}a_i\otimes \cdots \otimes a_{m-1}\otimes \zeta(a_m))\\
&\qquad-\sum_{i=1}^m(-1)^i\theta(f\circ \zeta\otimes a_0\otimes \cdots \otimes a_{i-1}a_i\otimes \cdots \otimes a_m )\\
&\qquad-\sum_{i=1}^m(-1)^i\zeta(\theta(f\otimes a_0\otimes \cdots\otimes a_{i-1}a_i\otimes \cdots \otimes a_m))\\
&\qquad+(-1)^{m+1}\theta(f\otimes\zeta(a_0)\otimes a_1\otimes\cdots\otimes a_{m-1})a_m\\
&\qquad+(-1)^{m+1}\sum_{j=1}^{m-1}\theta(f\otimes a_0\otimes \cdots \otimes \zeta(a_j)\otimes \cdots \otimes a_{m-1})a_m \\
&\qquad -(-1)^{m+1}\theta(f\circ \zeta\otimes a_0\otimes \cdots\otimes a_{m-1})a_m\\
&\qquad-(-1)^{m+1}\zeta(\theta(f\otimes a_0\otimes \cdots \otimes a_{m-1})a_m)+(-1)^{m+1}\theta(f\otimes a_0\otimes \cdots \otimes a_{m-1})\zeta(a_m).
\end{align*} We finally reorder the terms in this expression:
\begin{align*}
\partial (\alpha_m(\theta))(f\otimes \ul{a})&=\sum_{j=1}^{m-1}\Big[\theta(f\cdot a_0\otimes a_1\otimes \cdots \otimes \zeta(a_j)\otimes \cdots \otimes a_{m})\\
&\qquad\qquad+\sum_{i=1}^{j-1}(-1)^i\theta(f\otimes a_0\otimes \cdots \otimes a_{i-1}a_i\otimes \cdots \otimes \zeta(a_j)\otimes \cdots \otimes a_m)\\
&\qquad\qquad+(-1)^j\theta(f\otimes a_0\otimes \cdots \otimes a_{j-2}\otimes a_{j-1}\zeta(a_j)\otimes \cdots \otimes a_{m})\\
&\qquad\qquad+(-1)^{j+1}\theta(f\otimes a_0\otimes \cdots\otimes a_{j-1}\otimes \zeta(a_j)a_{j+1}\otimes \cdots \otimes a_m)\\
&\qquad\qquad+\sum_{i={j+2}}^m(-1)^i\theta(f\otimes a_0\otimes \cdots\otimes a_{j-1}\otimes \zeta(a_j)\otimes \cdots \otimes a_{i-1}a_i\otimes \cdots \otimes a_m)\\
&\qquad\qquad +(-1)^{m+1}\theta(f\otimes a_0\otimes \cdots \otimes \zeta(a_j)\otimes \cdots \otimes a_{m-1})a_m \Big]\\
&\quad -\Big[\theta((f\circ \zeta)\cdot a_0\otimes a_1\otimes \cdots \otimes a_{m})\\
&\qquad +\sum_{i=1}^m(-1)^i\theta(f\circ \zeta\otimes a_0\otimes \cdots \otimes a_{i-1}a_i\otimes \cdots \otimes a_m )\\
&\qquad +(-1)^{m+1}\theta(f\circ \zeta\otimes a_0\otimes \cdots\otimes a_{m-1})a_m\Big]\\
&\quad -\Big[\zeta(\theta(f\cdot a_0\otimes a_1\otimes\cdots\otimes a_m))\\
&\qquad +\sum_{i=1}^m(-1)^i\zeta(\theta(f\otimes a_0\otimes \cdots\otimes a_{i-1}a_i\otimes \cdots \otimes a_m))\\
&\qquad +(-1)^{m+1}\zeta(\theta(f\otimes a_0\otimes \cdots \otimes a_{m-1})a_m)\Big]\\
&\quad+\Big[\theta(f\cdot\zeta(a_0)\otimes a_1\otimes \cdots\otimes a_m)-\theta(f\otimes \zeta(a_0)a_1\otimes a_2\otimes\cdots\otimes a_m)\\
&\qquad +\sum_{i=2}^m(-1)^i\theta(f\otimes \zeta(a_0)\otimes\cdots\otimes a_{i-1}a_i\otimes\cdots\otimes a_m)\\
&\qquad +(-1)^{m+1}\theta(f\otimes\zeta(a_0)\otimes a_1\otimes\cdots\otimes a_{m-1})a_m\Big]\\
&\quad +\Big[\theta(f\cdot a_0\otimes a_1\otimes \cdots \otimes a_{m-1}\otimes \zeta(a_m))\\
&\qquad\qquad+(-1)^m\theta(f\otimes a_0\otimes \cdots \otimes a_{m-1}\zeta(a_m))\\
&\qquad\qquad+\sum_{i=1}^{m-1}(-1)^{i}\theta(f\otimes a_0\otimes \cdots \otimes a_{i-1}a_i\otimes \cdots \otimes a_{m-1}\otimes \zeta(a_m))\\
&\qquad\qquad+(-1)^{m+1}\theta(f\otimes a_0\otimes \cdots \otimes a_{m-1})\zeta(a_m)\Big]
\\
&=\sum_{j=1}^{m-1}\partial \theta(f\otimes a_0\otimes \cdots \otimes \zeta(a_j)\otimes \cdots \otimes a_m)-\partial \theta(f\circ \zeta\otimes \ul{a})-\zeta(\partial \theta(f\otimes \ul{a}))\\
&\qquad\qquad+\partial \theta(f\otimes \zeta(a_0)\otimes a_1\otimes \cdots\otimes a_m )+\partial \theta(f\otimes a_0\otimes \cdots \otimes a_{m-1}\otimes \zeta(a_m))\\
&=\alpha_{m+1}(\partial \theta)(f\otimes \ul{a}).
\end{align*} This completes the proof.
\end{proof}

\begin{cor}\label{cor:cocycle and coboundary}
The map $\alpha_m$ defines a map  from $E_m$ to $E_m$, again denoted by $\alpha_m.$
\end{cor}

\begin{proof} 
It follows from Lemma~\ref{lemma:cocycle and coboundary} that $\alpha_m$ sends cocycle to cocycle and coboundary to coboundary, that is, $\alpha_m$ defines a map from $E_m$ to $E_m$.
\end{proof}

\begin{lem}\label{lemma:alphas satisfy equations}
The map $\alpha_m:E_m\rightarrow E_m$ satisfies Equations \eqref{eq:c1} and \eqref{eq:c2} for $n=1$. 
\end{lem}

\begin{proof} Note that for $n=1$, the equations become
\begin{align*}
\theta\cdot\zeta(c)&=-\alpha_m(\theta\cdot c)+\alpha_m(\theta)\cdot c&\eqref{eq:c1}\\
\zeta(c)\cdot\theta&=c\cdot \alpha_m(\theta)-\alpha_m(c\cdot \theta)&\eqref{eq:c2}
\end{align*} for all $\theta\in E_m$ and all $c\in C.$ 

Now for any $c\in C$, $\theta\in E_m$, $f\in DC$ and $\ul{a}=a_1\otimes\cdots\otimes a_m\in C^{\otimes m}$ we have
\begin{align*}
(-\alpha_m(\theta\cdot c)&+\alpha_m(\theta)\cdot c)(f\otimes \ul{a})=-\alpha_m(\theta\cdot c)(f\otimes \ul{a})+\alpha_m(\theta)(c\cdot f\otimes \ul{a})\\
&=-\left(\sum_{j=1}^m (\theta\cdot c)(f\otimes a_1\otimes \cdots\otimes \zeta(a_j)\otimes \cdots\otimes a_m)-(\theta\cdot c)(f\circ \zeta\otimes \ul{a})-\zeta((\theta\cdot c)(f\otimes \ul{a}))\right)\\
&\quad +\left(\sum_{j=1}^m \theta(c\cdot f\otimes a_1\otimes \cdots\otimes \zeta(a_j)\otimes \cdots\otimes a_m)-\theta((c\cdot f)\circ \zeta\otimes \ul{a})-\zeta(\theta(c\cdot f\otimes \ul{a}))\right)\\
&\stackrel{\eqref{eq:actions E}}=\theta(c\cdot (f\circ \zeta)\otimes \ul{a})-\theta((c\cdot f)\circ \zeta\otimes \ul{a})\\
&\stackrel{\eqref{eq:actions compo}}=\theta(\zeta(c)\cdot f\otimes \ul{a})\\
&\stackrel{\eqref{eq:actions E}}=(\theta\cdot \zeta(c))(f\otimes \ul{a})
\end{align*} where in the last three equalities, we have used respectively \eqref{eq:actions E}, \eqref{eq:actions compo} then \eqref{eq:actions E} again, thus proving Equation \eqref{eq:c1} and 
\begin{align*}
(c\cdot \alpha_m(\theta)&-\alpha_m(c\cdot \theta))(f\otimes \ul{a})=c\,\alpha_m(\theta)(f\otimes \ul{a})-\alpha_m(c\cdot \theta)(f\otimes \ul{a})\\
&=\left(\sum_{j=1}^m c\,\theta(f\otimes a_1\otimes \cdots\otimes \zeta(a_j)\otimes \cdots\otimes a_m)-c\,\theta(f\circ \zeta\otimes \ul{a})-c\zeta(\theta(f\otimes \ul{a}))\right)\\
&\quad+\left(\sum_{j=1}^m (c\cdot \theta)(f\otimes a_1\otimes \cdots\otimes \zeta(a_j)\otimes \cdots\otimes a_m)-(c\cdot \theta)(f\circ \zeta\otimes \ul{a})-\zeta((c\cdot \theta)(f\otimes \ul{a}))\right)\\
&\stackrel{\eqref{eq:actions E}}=-c\,\zeta(\theta(f\otimes \ul{a}))+\zeta(c\,\theta(f\otimes \ul{a}))\\
&\stackrel{\hphantom{\eqref{eq:actions E}}}=\zeta(c)\,\theta(f\otimes \ul{a})\qquad\text{since $\zeta$ is a derivation}\\
&\stackrel{\hphantom{\eqref{eq:actions E}}}=\left(\zeta(c)\cdot \theta\right)(f\otimes \ul{a}),
\end{align*} using \eqref{eq:actions E} in the third equality, which proves Equation \eqref{eq:c2}.
\end{proof}

\begin{thm}\label{cor:phi surjective Es} Let $B=C\ltimes E_m$ be the trivial extension of $C$ by the $C$-$C$-bimodule $E_m=\Ext_C^m(DC,C)$, for some integer $m\pgq 0.$ Then the first Hochschild projection morphism $\varphi^1:\HH^1(B)\rightarrow \HH^1(C)$ is surjective.
\end{thm}

\begin{proof} By  Corollary~\ref{cor:technical condition}, $\varphi^1$ is surjective if, and only if, for any derivation $\zeta\in \Hom_k(C,C)$, there exists $\alpha\in\Hom_k(E^{1,n-1},E)=\End_kE_m$ that satisfies Equations  \eqref{eq:c1} and \eqref{eq:c2}, Equation  \eqref{eq:c3} being empty for $n=1$. Corollary  \ref{cor:cocycle and coboundary} shows that $\alpha_m$ is a well-defined map in $\End_kE_m$ and we have proved in Lemma~\ref{lemma:alphas satisfy equations} that it satisfies Equations  \eqref{eq:c1} and \eqref{eq:c2}. Therefore $\varphi^1$ is surjective.
\end{proof}

\section{Study of the kernel of $\varphi^1$}\label{sec:kernel}

We now return to the general situation of a trivial extension $B=C\ltimes E$ of a finite-dimensional algebra $C$ by an arbitrary $C$-$C$-bimodule $E$. In the cases where $\varphi^*$ is surjective, it is natural to ask what its kernel is. We answer this question for $\varphi^0$ and $\varphi^1$ in homological terms.

\begin{lem}\label{lem:phi0 ses} Let $B=C\ltimes E$ be  the trivial extension of $C$ by a $C$-$C$-bimodule $E$ which is symmetric over $Z(C)$. Then there exists a short exact sequence of vector spaces 
\[ 0\rightarrow \hh^0(B,E)\rightarrow\HH^0(B)\xrightarrow{\varphi^0}\HH^0(C)\rightarrow 0. \]
\end{lem}

\begin{proof}
Since $E$ is symmetric over $Z(C)$, it follows from Corollary~\ref{cor:technical condition} that $\varphi^0$ is surjective. Moreover, $\varphi^0:Z(B)\rightarrow Z(C)$ identifies with the restriction of $p$ to $Z(B)$, and its kernel is the subspace of elements in $E$ that are central in $B$, which is isomorphic to $\Hom_{B\da B}(B,E)=\hh^0(B,E)$.
\end{proof}

Following the notation in \cite{Saorin}, we let $\cale(E)$ denote the $k$-subspace of $\Hom_{C\da C}(E,C)$ consisting of all $C$-$C$-bimodule morphisms $f:E\rightarrow C$ such that, for any $x,$ $y$ in $E$, we have $f(x)y+xf(y)=0$. Identifying $E$ with $i(E)$ and taking into account the fact that the $C$-$C$-bimodule structure of $E$ is given by means of $q$, this condition can be written $qf(x)y+xqf(y)=0.$ 

We shall see in Corollary~\ref{lem: cale E trivial} that if $C$ is  a triangular algebra of global dimension at most two and if $E=\Ext_C^2(DC,C)$, we  have $\cale(E)=0.$ This is generally not the case, as Example~\ref{ex explicit ses} continued below shows.

\begin{example}\label{ex explicit ses continued} We  compute $\cale (E)$ in  Example~\ref{ex explicit ses}. Any morphism $f\in\Hom_{C\da C}(E,C)$ is entirely determined by the $4$-tuple $(f(a_1-\bar{a}_0),f(a_0+\bar{a}_1),f(a_0\bar{a}_0),f(a_1\bar{a}_1))\in C^4.$ Since $f$ is a morphism of $C$-$C$-bimodules, we must have $f(a_i+(-1)^i\bar a_{i+1})\in e_iCe_{i+1}=k\alpha_i$  and $f(a_{i}\bar a_i)\in e_iCe_i=ke_i$ for $i\in\mathbb{Z}/2\mathbb{Z}.$ 
Moreover,
\begin{align*}
f(\alpha_{i+1}\cdot(a_i+(-1)^i\bar a_{i+1}))&=f(q(\alpha_{i+1})(a_i+(-1)^i\bar a_{i+1}))=f(a_{i+1}(a_i+(-1)^i\bar a_{i+1}))\\&=f((-1)^ia_{i+1}\bar a_{i+1})\\
\text{and }f(\alpha_{i+1}\cdot(a_i+(-1)^i\bar a_{i+1}))&=\alpha_{i+1}f(a_i+(-1)^i\bar a_{i+1})=0 \text{ (since $\alpha_{i+1}\alpha_i=0$ in $C$).}
\end{align*} Therefore $f(a_i\bar a_i)=0$ for $i\in\{0,1\}.$ 

Now let $f\in\Hom_{C\da C}(E,C)$ be defined by $f(a_i+(-1)^i\bar a_{i+1})=\lambda_i \alpha_i$ with $\lambda_i\in k.$  We want a necessary and sufficient condition for $f$ to be in $\cale (E),$ that is, to satisfy the relation  $qf(x)y+xqf(y)=0$ for all $x,y$ in $E.$ Since $f(a_i\bar a_i)=0$, paths of length three vanish in $B$ and $f$ and $q$ are morphisms of $C$-$C$-bimodules, it is enough to check this when $x$, $y$ are in the set $\{a_1-\bar{a}_0,a_0+\bar{a}_1\}$ and when the source of $y$ is target of $x,$ that is, $x$ and $y$ are different. We have 
\begin{align*}
qf(a_{i}+(-1)^i\bar a_{i+1})&(a_{i+1}-(-1)^i\bar a_{i})+(a_{i}+(-1)^i\bar a_{i+1})qf(a_{i+1}-(-1)^i\bar a_{i})\\&=\lambda_{i}a_{i}(a_{i+1}-(-1)^i\bar a_{i})+(a_{i}+(-1)^i\bar a_{i+1})\lambda_{i+1}a_{i+1}\\&=(-1)^i\left(-\lambda_{i}a_{i}\bar a_{i}+\lambda_{i+1}\bar a_{i+1}a_{i+1}\right)\\&=(-1)^i(\lambda_{i+1}-\lambda_{i})a_{i}\bar a_{i}
\end{align*} for $i\in\{0,1\}$
 so that $f\in\cale (E)$ if, and only if, $\lambda_0=\lambda_1.$

Finally, $\cale (E)=k\zeta$ where $\zeta:E\rightarrow C$ is the morphism of $C$-$C$-bimodules defined on our basis of $E$ by $\zeta(a_0+\bar a_1)=\alpha_0$, $\zeta(a_1-\bar a_0)=\alpha_1$ and $\zeta(a_i\bar a_i)=0$ for $i\in\{0,1\}.$
\end{example}

\begin{lem}\label{lem ker connecting morphism}
Let $B=C\ltimes E$ be the trivial extension  of $C$ by a $C$-$C$-bimodule $E$ and let $\delta_B^1:\hh^1(B,C)\rightarrow \hh^2(B,E)$ be the connecting morphism. Assume that $\delta^{0,1}=0$. Then 
\[ \ker\delta_B^1=\HH^1(C)\oplus \cale(E). \]
\end{lem}

\begin{proof} We apply the results in Proposition~\ref{prop:results CMRS}
 with $X=q(C)\cong C$ and $X=E$, which gives the identities
\begin{align*}
\hh^1(B,C)&=\hh^0(\calc^1(C))\oplus\hh^1(\calc^0(C))= \Hom_{C\da C}(E,C)\oplus \HH^1(C)\\
\hh^2(B,E)&=\hh^0(\calc^2(E))\oplus\hh^1(\calc^1(E))\oplus\hh^2(\calc^0(E))\\
&=\Hom_{C\da C}(E\otimes _C E,E)\oplus\Ext^1_{C\da C}(E,E)\oplus \hh^2(C,E).
\end{align*} 

Moreover,  $\delta_B^1=\delta^{1,0}+\delta^{0,1}$, where  the map $\delta^{1,0}\colon \Hom_{C\da C}(E,C)\rightarrow\Hom_{C\da C}(E\otimes _C E,E)$ is given by 
\[ \delta^{1,0}(f)= \id_E\otimes_C f+f\otimes_C\id_E\] for any  $f\in\Hom_{C\da C}(E,C)$ and $\delta^{0,1}:\HH^1(C)\rightarrow\Ext^1_{C\da C}(E,E)$ vanishes by assumption. Therefore $\delta_B^1=\delta^{1,0}.$  Moreover, by definition, $\ker \delta^{1,0}=\cale(E).$ The result then follows.
\end{proof}

\begin{thm}\label{thm:ker phi for trivial extension}
Let $B=C\ltimes E$ be the trivial extension  of $C$ by a $C$-$C$-bimodule $E$ and assume that $E$ is  symmetric over $Z(C)$. Assume that the Hochschild projection morphism $\varphi^1$ is surjective. Then we have a short exact sequence 
\[ 0\rightarrow \hh^1(B,E)\oplus \cale(E) \rightarrow \HH^1(B)\xrightarrow{\varphi^1}\HH^1(C)\rightarrow 0. \]
\end{thm}

\begin{proof}
Since $E$ is  symmetric over $Z(C)$, it follows from \cite[Proposition 3.3]{CMRS} that the connecting morphism $\hh^0(B,C)\rightarrow \hh^1(B,E)$ vanishes. Moreover, $\delta^{0,1}=0$ by Proposition~\ref{prop: SC for surjectivity} because $\varphi^1$ is surjective. Therefore there exists  a short exact sequence of vector spaces 
\[ 0\rightarrow \hh^1(B,E)\xrightarrow j\HH^1(B)\xrightarrow{\small\begin{pmatrix}\varphi^1\\\rho\end{pmatrix}}\HH^1(C)\oplus \cale(E)\rightarrow 0 .\] Let $(s,t)$ be a section of the surjection $\small\begin{pmatrix}\varphi^1\\\rho\end{pmatrix}$, so that $s$ is a section of $\varphi^1$, $t$ is a section of $\rho$ and we have the relations $\varphi^1 t=0=\rho s.$ Then $\ker \varphi^1=\im j\oplus \im t.$ Indeed, if $b\in\ker \varphi^1$ then $b=(b-t\rho(b))+t\rho(b)$ with $b-t\rho(b)\in\ker \varphi^1\cap\ker\rho=\ker\small\begin{pmatrix}\varphi^1\\\rho\end{pmatrix}=\im j $, and it is easy to check that $\im j\cap \im t=0.$

Finally, we have $\ker\varphi^1\cong \hh^1(B,E)\oplus \cale(E)$ as required.
\end{proof}

\begin{remark} In case $E=DC$, the exact sequence in Theorem~\ref{thm:ker phi for trivial extension} follows from \cite[Theorem 5.7]{CMRS}
\end{remark}

\begin{example}\label{ex explicit ses continued again}
Let us continue Examples \ref{ex explicit ses} and \ref{ex explicit ses continued}. We keep the same notation. The expression of $\varphi$ in our basis $\{[u_0],[u_1],[v_0],[v_1]\}$ shows that $\ker \varphi=\spn\{[u_0+v_0],[u_1],[v_1]\}.$

It was shown in \cite[Proposition 4.4]{EH} that $Z(C)\cong \HH^0(C)\cong k$. In particular, $E$ is symmetric over the centre $Z(C)$.
Therefore the theorem above and the results in the previous examples show that $\dim_k\hh^1(B,E)=\dim_k\HH^1(B)-\dim_k\HH^1(C)-\dim_k\cale(E)=4-1-1=2.$ Moreover, $u_0+v_0$ and $u_1+v_1$ are 
 derivations from $B$ to $E$ whose cohomology classes are linearly independant, therefore $\{[u_0+v_0],[u_1+v_1]\}$ is a basis of $\hh^1(B,E)$.
This can in fact be checked directly in this example 
using the minimal projective resolution of the $B$-$B$-bimodule $B$ given in \cite[1.1]{ST} or the bar resolution of $B$.

Note that, using Remark~\ref{rk: description projection on cale E}, the map $\rho$ is given by $\rho([u_0])=\zeta=\rho([v_1])$ and $\rho([u_1])=-\zeta=\rho([v_0])$ and that $\ker\varphi\cap\ker\rho=\spn\{[u_0+v_0],[u_1+v_1]\}=\hh^1(B,E).$
\end{example}

The statements that follow, which give decompositions of $\hh^1(B,E)$, were proved in \cite{ARS} under the assumption that $B$ is cluster-tilted with $C$ tilted and $E=\Ext^2_C(DC,C)$, but they remain valid in our more general situation. We include the proofs for the benefit of the reader. Recall that $\Der_0(B,E)$ denotes the group of normalised derivations from $B$ to $E$, see paragraph \ref{subsec:Hochschild cohomology}.

 \begin{lem}\label{lem decomp normalised}
 Let $B$ be the trivial extension of an algebra $C$  by a $C$-$C$-bimodule $E$. Then 
\[\Der_0(B,E) \cong \Der_0(C,E) \oplus \End_{C\da C}E.\]
\end{lem}

\begin{proof} Let $\delta $ be an element in $ \Der_0(B,E)$ and define  $k$-linear maps $d: C \rightarrow E, \, \, f: E \rightarrow E$ by setting 
\begin{align*}
 d(c)&= \delta(c,0) \text{   for all } c \in C\\
f(x)&= \delta(0,x) \text{   for all } x \in E
\end{align*}
that is, $d= \delta|_C$ and $f= \delta|_E$. We first show that $d \in \Der_0(C,E)$. Let $c,c'\in E$, then 
\[d(cc')= \delta(cc',0)= \delta ((c,0) (c',0))= (c,0) \delta(c',0) + \delta(c,0)(c',0)= cd(c') +d(c) c.\]

 Moreover, $d(e_i)= \delta (e_i,0)=0$ for every primitive idempotent $e_i$ of $B$.

 Next, we prove that $ f \in \End_{C\da C}E$. Let $ c \in C, \, x \in E,$ then 
\[ f(cx)= \delta(0,cx)=\delta((c,0)(0,x)) = (c,0)\delta(0,x) + \delta(c,0) (0,x)= c f(x) + d(c)x= c f(x),\]
because $d(c) \in E$ and $E^2=0$. Similarly, $f(xc)=f(x)c$.

Because $\delta=d+f$, this proves that $\Der_0(B,E)=\Der_0(C,E)+\End_{C\da C}E.$ Now $\Der_0(C,E)\subseteq \Hom_k(C,E)$ and $\End_{C\da C}E\subseteq \End_k(E).$ Therefore $\Hom_k(C,E)\cap\End_kE=0$, which implies the statement.
\end{proof}

 \begin{prop}\label{prop:decomp H1 B E}
 Let $B$ be the trivial extension of an algebra $C$  by a $C$-$C$-bimodule $E$. Then 
\[\hh^1(B,E) \cong \hh^1(C,E) \oplus \End_{C\da C}E.\]
\end{prop}

\begin{proof} We claim that the direct sum decomposition $\Der_0(B,E)=\Der_0(C,E) \oplus \End_{C\da C}E$ of Lemma~\ref{lem decomp normalised} above induces on the level of the inner derivations the decomposition $\Inn_0(B,E)= \Inn_0(C,E) \oplus 0$, that is, if $\delta \in \Inn_0(B,E)$ then $d= \delta|_C \in \Inn_0(C,E)$ and $f= \delta|_E=0$.

Since $\delta \in \Inn_0(B,E)$, there exists $x_0\in C$ such that $\delta = [x_0,-]$, that is, for all $(c,x) \in B$ we have \[\delta(c,x)= (0,x_0)(c,x) - (c,x)(0,x_0)= (0, x_0c-cx_0)\] (note that we identify $E$ and $i(E)$).

It follows immediately that $\delta|_C:c\mapsto x_0c-cx_0$ is an inner derivation from $C$ to $E$ and that $\delta|_E=0.$

The statement then follows by taking quotients:
\[\hh^1(B,E) \cong \frac{\Der_0(B,E)}{\Inn_0(B,E)} \cong \frac{\Der_0(C,E) \oplus \End_{C\da C}E}{\Inn_0(C,E) \oplus 0 } \cong \hh^1(C,E) \oplus \End_{C\da C}E.\qedhere\]
\end{proof}

\section{Relation extensions}\label{sec:relation extensions}


From now on, we assume that the base field $k$ is algebraically closed, so that all our algebras are given by bound quivers.  Let $C$ be an algebra of global dimension at  most two. We recall from \cite{ABS1} that the \emph{relation extension} of $C$ is the trivial extension of $C$ by $E_2=\Ext^2_{C}(DC,C)$ equipped with its natural 
$C$-$C$-bimodule structure (as in \eqref{eq:actions E}). The best known class of relation extensions is provided by the 
cluster-tilted algebras: let $C$ be a tilted algebra, then the relation extension $B=C \ltimes E_2$ is cluster-tilted and every cluster-tilted algebra arises in this way, see \cite[Theorem 3.4]{ABS1}.
The purpose of this section is to prove that the main result of \cite{ABIS}, proved in case $B$ is cluster-tilted, remains valid for any relation extension of a triangular algebra of global dimension at most two.

The quiver of a relation extension is easily computed by means of \cite[Theorem 2.6]{ABS1},  which we restate for future reference. Recall that, if $C=kQ_C/I$ is a bound quiver algebra, then a \emph{system of relations} $R$ for $C$ is a subset of $\bigcup_{x,y \in (Q_C)_0}(e_xIe_y)$ such that $R$, but no proper subset of $R$, generates $I$ as a two-sided ideal.

\begin{lem}\label{lem quiver rel ext}
Let $C=kQ_C/I$ be an algebra of global dimension at most two and let $R$ be a system of relations for $C$. The quiver $Q_B$ of the relation extension $B= C \ltimes E_2$ is constructed as follows:
\begin{enumerate}
\item[(a)] $(Q_B)_0=(Q_C)_0$.
\item[(b)] For $x,y \in (Q_B)_0$, the set of arrows in $Q_B$ from $x$ to $y$ equals the set of arrows in $Q_C$ from $x$ to $y$ (called the \emph{old} arrows) plus $\card(R \cap e_y I e_x)$ additional arrows (called the \emph{new} arrows).
\end{enumerate}
\end{lem}

Thus, if $C$ is not hereditary, then $B$ is not triangular so we may define a  potential on its quiver, see \cite{DWZ}.  Let $R=\{ \rho_1, \cdots, \rho_t \}$. Because of Lemma~\ref{lem quiver rel ext}, to each relation $\rho_i$ from $x_i$ to $y_i$, say, there corresponds a new arrow $\alpha_i: y_i \rightarrow x_i$. The \emph{Keller potential} on $Q_B$ is $W_B=\sum_{i=1}^t \rho_i \alpha_i$. 
Potentials are considered up to cyclic permutations: two potentials are \emph{cyclically equivalent} if their difference lies in the linear span of all elements of the form $\alpha_1 \alpha_2 \cdots \alpha_m -\alpha_m\alpha_1 \alpha_2 \cdots \alpha_{m-1}$, where $\alpha_1 \alpha_2 \cdots \alpha_m $ is a cycle. For a given arrow $\beta$, the \emph{cyclic partial derivative} $\partial_{\beta}$ of a potential $W$ is defined on each cyclic summand $\beta_1 \cdots \beta_s$ of $W$ by $\partial_{\beta}(\beta_1 \cdots \beta_s)= \sum_{i: \, \beta \ne \beta_i} \beta_{i+1}\cdots \beta_s\beta_1 \cdots \beta_{i-1}$. In particular, the cyclic derivative is invariant under cyclic permutation. The \emph{Jacobian algebra} $\mathcal{J}(Q_B,W_B)$ is the one given by the quiver $Q_B$ bound by the ideal generated by all the cyclic partial derivatives $\partial_{\alpha} W_B$ of the Keller potential $W_B$ with respect to each arrow $\alpha \in (Q_B)_1$, see \cite{Keller}. We now derive a system of relations for $B$ (compare \cite[Proposition 3.19]{BFPPT}).

\begin{lem}
Let $C$ be an algebra of global dimension at most two, B its relation extension and $W_B$ the Keller potential, then $B \cong \mathcal{J}(Q_B,W_B)/J$ where $J$ is the square of the ideal generated by the new arrows.
\end{lem}
\begin{proof}
It was shown by Keller in \cite[Theorem 6.12 (a)]{Keller} that $ \mathcal{J}(Q_B,W_B)$ is isomorphic to the endomorphism algebra of the tilting object $C$ in the (Amiot's generalised) cluster category associated with the algebra $C$. On the other hand, Amiot proved in \cite[Proposition 4.7]{A} that this endomorphism algebra is isomorphic to the tensor algebra of the bimodule $E_2=\Ext_C^2(DC,C)$. Also, she proved that the quiver of the tensor algebra of $E_2$ is isomorphic to the quiver of $\mathcal{J}(Q_B,W_B)$. Taking now the quotient of the tensor algebra by the two-sided ideal $J$ generated by all tensor powers $E_2^{\otimes  m}$ with $ m \pgq 2$, we get exactly the algebra $B$. But now $J$ is the square of the ideal generated by the new arrows.
\end{proof}

Thus $B \cong kQ_B/I_B$ where $Q_B$ is computed as in Lemma~\ref{lem quiver rel ext}, while $I_B$ is generated by all the cyclic partial derivatives $\partial_{\alpha} W_B$ with $\alpha \in (Q_B)_1$, and  all the paths in $Q_B$ containing at least two new arrows.

\begin{example}\label{ex:from ABS1} We borrow this example from \cite[Example 2.7]{ABS1}. Let $C$ be given by the quiver 
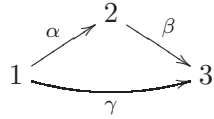
\[\xymatrix@R=10pt{ & 2 \ar[rd]^{\beta} & \\
1\ar[ru]^{\alpha} \ar@/_/[rr]_{ \gamma } &  & 3
}
\] bound by $\alpha \beta =0$. Applying Lemma~\ref{lem quiver rel ext}, we  get a new arrow $\delta: 3 \rightarrow 1$. The quiver $Q_B$ of $B= C \ltimes E_2$ is therefore
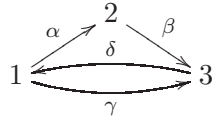
\[\xymatrix@R=10pt{ & 2 \ar[rd]^{\beta} & \\
1\ar[ru]^{\alpha} \ar@/_/[rr]_{ \gamma } &  & 3  \ar@/_5pt/[ll]_{\delta}
}
\] We have $W_B= \alpha \beta \delta$ while $J=< \delta>^2 $. Thus $B$ is given by the quiver $Q_B$ bound by $\delta \alpha=0, \, \alpha \beta=0, \, \beta \delta=0$ and $\delta  \gamma  \delta=0$. Note that $B$ is not a cluster-tilted algebra.
\end{example}

\begin{lem}\label{lem caract cale E} Let $B$ be the trivial extension of an algebra $C$ by a $C$-$C$-bimodule $E$. Then $f \in \cale(E)$ if and only if there exists a derivation $\bar{f}: B \rightarrow C$ which restricts to $f$ and vanishes on $C$.
\end{lem}
\begin{proof}
Assume $f \in \cale(E)$ and define $\bar{f}: B \rightarrow C$ by $\bar{f}(c,x)=f(x)$ for $c \in C, \, x \in E$. Clearly, $\bar{f}$ vanishes on $C$. We prove that $\bar{f}$ is a derivation. Let $(c,x),\, (c',x') \in B$, then 
\begin{eqnarray*} (c,x) \bar{f}(c',x') + \bar{f}(c,x) (c',x') & = & (c,x) f(x')+ f(x) (c',x') \\
 &=& (cf(x')+f(x) c',x f(x') + f(x) x' ) \\
&=&(cf(x') + f(x) c',0)
 \end{eqnarray*} 
because $f \in \cale(E)$. Using that $f$ is a morphism of bimodules, we get

\begin{eqnarray*} (c,x) \bar{f}(c',x') + \bar{f}(c,x) (c',x') & = & (f(cx'+xc'),0) \\
&=& \bar{f}(cc',cx'+xc') \\
&=& \bar{f}((c,x)(c',x')).
\end{eqnarray*}

This establishes our claim.

Conversely, let $f: B \rightarrow C$ be a derivation which vanishes on $C$. Then $f$ restricts to a $k$-linear map from $E$ to $C$. We claim that it is in fact a morphism of $C$-$C$-bimodules. Let $c \in C, \, x \in E$. Viewing both as elements of $B$ and using that $f$ is a derivation on $B$, we get
$$ f(cx)=f(c)x+cf(x)=cf(x)$$
because $f(c)=0$. Similarly, $f(xc)=f(x)c$. There remains to prove that $f \in \cale(E)$. Let $x,x' \in E$. Since $f$ is a derivation, we have 
$$ xf(x')+f(x)x'=f(xx')=f(0)=0$$
because $xx' \in E^2=0$.
 \end{proof}

\begin{lem} \label{lem: cale E trivial}
Let $B$ be the relation extension of a triangular algebra $C$ of global dimension at most two by the $C$-$C$-bimodule $E_2$. Then $\cale(E_2)=0$.
\end{lem}

\begin{proof}
Let $f \in \cale(E_2)$. Because of Lemma~\ref{lem caract cale E}, $f$ extends to a derivation $\bar{f}: B \rightarrow C$ which vanishes on $C$. Let $ \gamma : x \rightarrow y$ be a new arrow, that is, a generator of $E_2$ as $C$-$C$-bimodule. Then $\bar{f}$ sends $ \gamma  \in e_x E_2 e_y$ into $e_x C e_y$: indeed, $\bar{f}( \gamma )=\bar{f}(e_x  \gamma  e_y)=e_x \bar{f}( \gamma ) e_y$ because $\bar{f}(e_x)=\bar{f}(e_y)=0$, both elements $e_x,e_y$ lying in $C$. Now, the existence of a new arrow $ \gamma : x \rightarrow y$ means that there exists a path from $y $ to $x$ (actually, there exists a relation from $y$ to $x$) in $C$, that is, formed entirely of old arrows. Because $C$ is triangular, we have $e_x C e_y=0$ and so $\bar{f}=0$. This finishes the proof.
\end{proof}

\begin{cor}
Let $B$ be the relation extension of a triangular algebra $C$ of global dimension at most two by the $C$-$C$-bimodule $E_2$. Then $\ker\delta_B^1\cong \HH^{1}(C).$
\end{cor}

We are now able to prove the main result of  this section which is our Theorem~\ref{thmIntro:case triangular}.

\begin{thm} \label{thm ses triangular case} Let $B$ be the relation extension of a triangular algebra $C$ of global dimension at most two by the $C$-$C$-bimodule $E_2$. Then we have short exact sequences
\begin{enumerate}[(a)]
\item $0\rightarrow \hh^0(B,E_2)\longrightarrow \HH^0(B) \xrightarrow{\varphi^0} \HH^{0}(C)\rightarrow 0.$
\item $0\rightarrow \hh^1(B,E_2)\longrightarrow \HH^1(B) \xrightarrow{\varphi^1} \HH^{1}(C)\rightarrow 0.$
\end{enumerate}
\end{thm}
\begin{proof} 
Part (a) follows from Lemma~\ref{lem:phi0 ses}, since $E_2$ is symmetric over $Z(C)=k$. The short exact sequence in (b) follows from Lemma~\ref{lem: cale E trivial} and Theorem~\ref{thm:ker phi for trivial extension}.
\end{proof}

We give an application to cluster-tilted algebras.

\begin{cor}\label{cor: HH*}  Let $B$ be a cluster-tilted algebra and let $C$ be a tilted algebra such that $B=C\ltimes E_2.$ Then the algebra morphism $\varphi^*:\HH^*(B)\rightarrow \HH^*(C)$ is surjective. Moreover, there is an exact sequence 
\[ 0\rightarrow \hh^0(B,E_2)\oplus \hh^1(B,E_2)\rightarrow\HH^*(B)\xrightarrow{\varphi^*}\HH^*(C)\rightarrow 0. \]

\end{cor}

\begin{proof} By \cite{H2}, since $C$ is tilted, $\HH^n(C)=0$ for all $n\pgq 2$. The result then follows from Theorem~\ref{thm ses triangular case}.
\end{proof}

We return to the more general situation of relation extensions.

\begin{cor}\label{cor:decomp ker phi} Let $B$ be the relation extension of a triangular algebra $C$ of global dimension at most two by the $C$-$C$-bimodule $E_2$. Then we have a short exact sequence
\[ 0\longrightarrow \hh^1(C,E_2) \oplus \End_{C\da C}E_2  \longrightarrow \HH^1(B)\xrightarrow{\varphi^1} \HH^1(C)\rightarrow0. \]
\end{cor}
\begin{proof} This follows from Theorem~\ref{thm ses triangular case} and Proposition~\ref{prop:decomp H1 B E}.
\end{proof}

\begin{example}\label{ex:ABS1 continued} Let $C$ and $B$ be the algebras in Example~\ref{ex:from ABS1}.

We first claim that $\hh^1(C,E_2)=0$. It suffices to prove that $\Der_0(C,E_2)=0$. Indeed, for any old arrow $\xi \in \{ \alpha, \beta,  \gamma \}$ from $x$ to $y$, and $d \in \Der_0(C,E_2)$, we have $d(\xi)= e_x d(\xi)e_y \in e_x E_2 e_y$. Therefore there exists a path from $x$ to $y$ in $B$ passing through the new arrow $\delta$ and parallel to the arrow $\xi$. Now any such path is easily seen to be zero. Therefore $\Der_0(C,E_2)=0$ and so $\hh^1(C,E_2)=0$.

This implies that $\hh^1(B,E_2) \cong \End_{C\da C}E_2$. Now, the $C$-$C$-bimodule $E_2$ has simple top generated by the new arrow $\delta$ and is actually $4$-dimensional, with basis $\{ \delta, \delta  \gamma ,  \gamma  \delta,  \gamma  \delta  \gamma  \}$. Therefore $\End_{C\da C}E_2 \cong  k$ and so $\hh^1(B,E_2) \cong  k$. Because we also have $\HH^1(C) \cong  k$ as can be seen, for instance, using D.~Happel's sequence in \ref{subsec:Hochschild cohomology}, we get $\HH^1(B) \cong k^2$.
\end{example}

The following corollary gives a lower bound on the increase in the dimension of the Hochschild cohomology group (compare with \cite[4.3]{ABIS}). 

\begin{cor} Let $B$ be the trivial extension of an algebra $C$ by a $C$-$C$-bimodule $E$ which is symmetric over $Z(C)$. Assume that $E=\bigoplus_{i=1}^nE_i$ is a direct sum decomposition into indecomposable summands as a $C$-$C$-bimodule. Then 
\begin{enumerate}[(a)]
\item $\dim_k\HH^1(B)-\dim_k\HH^1(C)\pgq n$.
\item Equality holds in (a) if, and only if, $\hh^1(C,E)=0$, $\cale(E)=0$ and 
\[ \dim_k\Hom_{C\da C}(E_i,E_j)=
\begin{cases}
k&\text{if $i=j$}\\0&\text{if $i\neq j$}.
\end{cases}
 \]
\end{enumerate}
\end{cor}

\begin{proof}
\begin{enumerate}[(a)]
\item This follows from Corollary~\ref{cor:decomp ker phi} and from the fact that the $\id_{E_i}$ provide $n$ linearly independent elements in $\End_{C\da C}E.$
\item This follows from the facts that $\dim_k\End_{C\da C}E\pgq n$ and that equality holds if, and only if, $n=\dim_k\End_{C\da C}E+\dim_k\hh^1(C,E)+\dim_k\cale(E).$
\qedhere\end{enumerate}
\end{proof}

Even in the case of a relation extension, when $\varphi^1$ is surjective, the higher Hochschild projection morphisms need not be surjective, as Example~\ref{ex:phi2 not surjective} below illustrates. In order to give this example, we need some background on minimal projective resolutions.

 It is well-known that  $\hh^i(C,E)=\Ext^i_{C\da C}(C,E)$. This means in particular that in order to compute the Hochschild cohomology spaces, we can compute the Hochschild cohomology of the complex $\left(\Hom_{C\da C}(\sfp^*,E),\Hom_{C\da C}(d^*,E)\right)$, where $(\sfp^*,d^*)$ is any projective resolution of the $C$-$C$-bimodule $C.$ Traditionally, we choose $\sfp^*=\sfbar^*$, the bar resolution, given by $\sfbar^n=C^{\otimes{n+2}}$ with differential $d^n:\sfbar^{n+1}\rightarrow\sfbar^{n}$ defined by 
\[ c_0\otimes\cdots\otimes c_{n+2}=\sum_{i=0}^{n+1}(-1)^ic_0\otimes \cdots\otimes c_ic_{i+1}\otimes \cdots\otimes c_{n+2} .\] Applying $\Hom_{C\da C}(-,E)$ to this resolution and using the isomorphism $\Hom_{C\da C}(C\otimes V\otimes C,E)\cong \Hom_k(V,E)$ for any vector space $V$ gives the complex that we have been using throughout this paper.

However, for the computation of higher Hochschild cohomology groups, the bar resolution is often too large, and it is therefore necessary to choose for $\sfp^*$ a minimal projective resolution of $C$. If $C=kQ/I$ is a bound quiver algebra, it was proved in \cite{H} that if $\sfp^*$ is a minimal projective resolution of the $C$-$C$-bimodule $C,$ then $\sfp^n=\bigoplus_{x,y\in Q_0}\left(Ce_x\otimes e_yC\right)^{\dim_k\Ext_C^n(S_C(x),S_C(y))}$. However, this result does not give any information on the maps $d^n:\sfp^n\rightarrow \sfp^{n-1}$, which are needed in order to compute the Hochschild cohomology spaces explicitly. Finding these maps is in fact a difficult problem, solved only in some special cases such as stacked monomial algebras  in \cite{GS2}. However,  we shall only  compute $\HH^1(C)$ and $\HH^2(C)$, and a minimal projective resolution $(\sfp^*,d^*)$ with maps up to $(\sfp^3,d^3)$ has been described by  in \cite{GS1}. Let $R$ be  a system of relations for $C=kQ/I$. Set 
\begin{align*}
&\sfp^0=\bigoplus_{x\in Q_0}Ce_x\otimes e_xC, \\ &\qquad d^0(e_x\otimes e_x)=e_x\\
&\sfp^1=\bigoplus_{a\in Q_1}Ce_{\mo(a)}\otimes e_{\mt(a)}C, \\ &\qquad d^1(e_{\mo(a)}\otimes e_{\mt(a)})=a\otimes e_{\mt(a)}-e_{\mo(a)}\otimes a=ae_{\mt(a)}\otimes e_{\mt(a)} -e_{\mo(a)}\otimes e_{\mo(a)}a\\
&\sfp^2=\bigoplus_{r\in R}Ce_{\mo(r)}\otimes e_{\mt(r)}C, \\ &\qquad d^2({e_{\mo(r)}}\otimes {e_{\mt(r)})}=\sum_{i=1}^n\sum_{j=1}^{s_i}c_i\alpha_{1,i}\cdots \alpha_{j-1,i}e_{\mo(\alpha_{j,i})}\otimes e_{\mt(\alpha_{j,i})} \alpha_{j+1,i}\cdots\alpha_{s_i,i}
\end{align*}  for any $r=\sum_{i=1}^nc_i\alpha_{1,i}\cdots \alpha_{s_i,i}\in R$. The descriptions of $\sfp^3$ and $d^3$ are a little more technical. It was proved in \cite{GSZ} that there exist sets $g^n$ for $n\pgq 0$ such that $g^0=\{e_x;x\in Q_0\}$, $g^1=Q_1,$ $g^2=R$ and for $n\pgq 3$, the elements in $g^n$ are  elements $x\in \cup_{i,j\in Q_0}e_ikQ e_j$ satisfying $x=\sum_{y\in g^{n-1}}y\alpha_y=\sum_{z\in g^{n-2}}z\beta_z$ for unique $\alpha_y$ and $\beta_z$ in $kQ$ having special properties and such that $\left(\sfr^*,\delta^*\right)$ is a minimal projective right resolution of the $C$-module $C/\rad(C)$, where $\sfr^n=\bigoplus_{x\in g^n}e_{\mt(x)}C$ and $\delta^n:\sfr^n\rightarrow \sfr^{n-1}$ is defined by $\delta^n(e_{\mt(x)})=\sum_{y\in g^{n-1}}\alpha_ye_{\mt(x)}$.  

E.L.~Green and N.~Snashall then defined $\sfp^3=\bigoplus_{x\in g^3}Ce_{\mo(x)}\otimes e_{\mt(x)}C$. 
 Moreover, if $x\in g^3$, then $x$ may be written uniquely $x=\sum_{r\in R}r\alpha_r=\sum_{r\in R}\gamma_rr\gamma_r'$ where $\alpha_r$ and $\gamma_r$ are in $kQ^+.$ The map $d^3$ is then given by $d^3(e_{\mo(x)}\otimes e_{\mt(x)})=\sum_{r\in R}(e_{\mo(r)}\otimes e_{\mt(r)}\alpha_r-\gamma_re_{\mo(r)}\otimes e_{\mt(r)}\gamma_r').$ Finally, they proved in \cite[Theorem 2.9]{GS1} that  the sequence \[\sfp^3\xrightarrow{\ d^3\ }\sfp^2\xrightarrow{\ d^2\ }\sfp^1\xrightarrow{\ d^1\ } \sfp^0\xrightarrow{\ d^0\ } C\rightarrow 0\] forms part of a minimal projective resolution of the $C$-$C$-bimodule $C.$

\begin{example}\label{ex:phi2 not surjective}
We return to Example~\ref{ex:ABS1 continued} and use minimal projective resolutions to compute $\HH^2(C)$, $\HH^2(B)$ and determine $\varphi^2.$

Using the results described above, a minimal projective $C$-$C$-bimodule resolution of $C$ is given by 
\[ 0\rightarrow \sfp^2=Ce_1\otimes e_3C\xrightarrow{d^2}\sfp^1=(Ce_1\otimes e_2C)\oplus(Ce_2\otimes e_3C)\oplus (Ce_1\otimes e_3C)\xrightarrow{d^1}\sfp^0=\bigoplus_{i=1}^3 Ce_i\otimes e_i C\rightarrow0.\] Therefore $\HH^*(C)$ is the cohomology of the complex 
\[ 0\rightarrow \Hom_{C\da C}(\sfp^0,C)\xrightarrow{d_1}\Hom_{C\da C}(\sfp^1,C)\xrightarrow{d_2}\Hom_{C\da C}(\sfp^2,C)\rightarrow 0. \]
A morphism $f\in \Hom_{C\da C}(\sfp^0,C)$ is entirely determined by $f(e_i\otimes e_i)=\lambda_ie_i\in e_iCe_i$ for $i=1,2,3$, where $\lambda_i\in k.$ Therefore 
\begin{align*}
d_1f(e_1\otimes e_2)&=(\lambda_2-\lambda_1)\alpha\\
d_2f(e_2\otimes e_3)&=(\lambda_3-\lambda_2)\beta\\
d_3f(e_1\otimes e_3)&=(\lambda_3-\lambda_1) \gamma 
\end{align*} so that $\dim_k \ker d_1=1$ and $\dim_k \im d_1=2.$ A morphism $g\in \Hom_{C\da C}(\sfp^1,C)$ is entirely determined by $g(e_1\otimes e_2)=\mu_1\alpha\in e_1Ce_2$, $g(e_2\otimes e_3)=\mu_2\beta\in e_2Ce_3$, $g(e_1\otimes e_3)=\mu_3 \gamma \in e_1Ce_3$,  where $\mu_i\in k.$ Therefore $d_2g(e_1\otimes e_3)=0$ and we have $\dim_k\ker d_2=3$ and $\im d_2=0.$

It then follows that $\dim_k \HH^0(C)=\dim_k \HH^1(C)=\dim_k \HH^2(C)=1$ and $\dim_k \HH^n(C)=0$ for $n\pgq 3$, and moreover that a basis for $\HH^1(C)$ (respectively $\HH^2(C)$) is given by the cohomology class of $\zeta_C$ (respectively $\xi_C$) determined by $\zeta_C(e_1\otimes e_2)=\alpha$, $\zeta_C$ sends $e_2\otimes e_3$ and $e_1\otimes e_3$ to $0$ and  $\xi_C(e_1\otimes e_3)= \gamma $.

The computation for $B$ is more complicated. The sets $g^n$ for $n=0,1,2,3$ are given by
\begin{align*}
g^0&=\{e_i;i=1,2,3\}\\
g^1&=\left\{\alpha;\beta; \gamma ;\delta\right\}\\
g^2&=\left\{g_1^2:=\alpha\beta;g_2^2:=\delta\alpha;g_3^2:=\beta\delta;g_4^2:=\delta \gamma \delta\right\}\\
g^3&=\{\alpha\beta\delta=g_1^2\delta=\alpha g_3^2;\ \beta\delta\alpha=g_3^2\alpha=\beta g_2^2;\ \delta\alpha\beta=g_2^2\beta=\delta g_1^2;\\
&\qquad \beta\delta \gamma \delta=g_3^2 \gamma \delta=\beta g_4^2;\ \delta \gamma \delta \gamma \delta=g_4^2 \gamma \delta=\delta \gamma  g_4^2;\ \delta \gamma \delta\alpha=g_4^2\alpha=\delta \gamma  g_1^2\}
\end{align*} so that the beginning of a minimal projective resolution of the $B$-$B$-bimodule $B$ is given by 
\begin{align*}
\cdots&\rightarrow \sfq^3=\left(\bigoplus_{i=1}^3(Be_i\otimes e_iB)\right)\oplus (Be_2\otimes e_1B)\oplus (Be_3\otimes e_1B)\oplus(Be_3\otimes e_2B)\\
&\xrightarrow{\partial ^3} \sfq^2=(Be_1\otimes e_3B)\oplus (Be_3\otimes e_2B)\oplus (Be_2\otimes e_3B)\oplus (Be_3\otimes e_1B)\\
&\xrightarrow{\partial ^2} \sfq^1=(Be_1\otimes e_2B)\oplus (Be_2\otimes e_3B)\oplus (Be_1\otimes e_3B)\oplus (Be_3\otimes e_1B)\\
&\xrightarrow{\partial ^0} \sfq^0=\bigoplus_{i=1}^3(Be_i\otimes e_iB)\rightarrow B\rightarrow 0
\end{align*} with 
\begin{xalignat*}{2}
&\partial ^2(e_1\otimes e_3)=e_1\otimes \beta+\alpha\otimes e_3&&\partial ^2(e_2\otimes e_1)=e_2\otimes \delta +\beta\otimes e_1\\
&\partial ^2(e_3\otimes e_1)=e_3\otimes  \gamma \delta+\delta\otimes \delta+\delta \gamma \otimes e_1&&\partial ^2(e_3\otimes e_2)=e_3\otimes \alpha_\delta\otimes e_2\\
&\partial ^3(e_1\otimes e_1)=e_1 \otimes \delta -\alpha\otimes e_1 &&\partial ^3(e_2\otimes e_2)=e_2 \otimes  \gamma \delta-\beta\otimes e_2 \\
&\partial ^3(e_3\otimes e_3)=e_3 \otimes \alpha-\beta\otimes e_3 &&\partial ^3(e_2\otimes e_1)=e_2 \otimes  \gamma \delta-\delta \gamma \otimes e_1 \\
&\partial ^3(e_3\otimes e_1)=e_3 \otimes \beta-\delta\otimes e_1 &&\partial ^3(e_3\otimes e_2)=e_3 \otimes \alpha-\delta \gamma \otimes e_2.
\end{xalignat*} Therefore the first Hochschild cohomology groups for $B$ are the cohomology groups of the complex 
\[ 0\rightarrow \Hom_{B\da B}(\sfq^0,B)\xrightarrow{\partial _1}\Hom_{B\da B}(\sfq^1,B)\xrightarrow{\partial _2}\Hom_{B\da B}(\sfq^2,B)\xrightarrow{\partial _3}\rightarrow \Hom_{B\da B}(\sfq^3,B)\rightarrow\cdots \] in which $\partial _i=\Hom_{B\da B}(\partial ^i,B)$ for all $i.$ A computation similar to that above shows that $\dim_k \im \partial _1=3$, $\dim_k\ker \partial _1=2$, $\im\partial _2=0$, $\dim_k\ker \partial _2=5$ and $\dim_k\ker\partial _3=2$ so that $\dim_k \HH^0(B)=\dim_k\HH^1(B)=\dim_k\HH^2(B)=2$, and  that  a basis for $\HH^1(B)$ (respectively $\HH^2(B)$) is given by $\left\{\zeta_B^1;\zeta_B^2\right\}$ (respectively $\left\{\xi_B^1;\xi_B^2\right\}$) determined by $\zeta_B^1(e_1\otimes e_2)=\alpha$, $\zeta_B^1(e_1\otimes e_3)= \gamma $, $\xi_B^1(e_1\otimes e_3)= \gamma \delta \gamma $ and $\xi_B^2(e_3\otimes e_1)=\delta$, and $\zeta_B^j$ and $\xi_B^j$ send all other generators of the corresponding projective to $0.$

We now check that $\varphi^2(\xi_B^j)=0$ for $j=1,2$ so that $\varphi^2=0.$ There exist comparison morphisms $\omega_B^n:\sfq^n\rightarrow \sfbar^n(B)$ and $\tau^n_B:\sfbar^n(B)\rightarrow \sfq^n$ between the minimal and the bar resolutions of $B$, and similarly for $C.$ Given a cocycle $f\in\Hom_{B\da B}(\sfq^n,B)$, we have $[f]=[\tau^n_B(f)]$ and for any cocycle $g\in\Hom_k(B^{\otimes n},B)\cong\Hom_{B\da B}(\sfbar^n(B),B)$ we have $[g]=[\tau^n_B(g)]$. Moreover, 
if $g\in\Hom_{B\da B}(\sfbar^n(B),B)$, then via the isomorphism $\Hom_k(B^{\otimes n},B)\cong\Hom_{B\da B}(\sfbar^n(B),B)$, we have $\varphi^n([g])=[pgq^{\otimes (n+2)}]$; indeed, if $\bar{g}\in \Hom_k(B^{\otimes n},B)$ is the corresponding linear map, we have, with arguments similar to those in the proof of Lemma~\ref{cor:cohom C embeds B}, 
\begin{align*}
pgq^{\otimes(n+2)}(c_0\otimes c_1\otimes\cdots\otimes c_n\otimes c_{n+1})&=p\left(q(c_0)g(1\otimes q(c_1)\otimes \cdots \otimes q(c_n)\otimes 1)q(c_{n+1})\right)\\
&=c_0p(\bar{g}(q(c_1)\otimes \cdots \otimes q(c_n)))c_{n+1}\\
&=c_0p\bar gq^{\otimes n}(c_1\otimes \cdots \otimes c_n)c_{n+1}
\end{align*} which is the element in $\Hom_{C\da C}(C^{\otimes(n+2)},C)$ corresponding to $p\bar gq^{\otimes n}\in\Hom_k(C^{\otimes n},C)$, the cocycle representing $\varphi^n([\bar g])$,  as required.
It then follows that $\varphi^n([f])$ is given by the cohomology class of the composition 
\[ \sfp^n\xrightarrow{\omega_C^n}\sfbar^n(C)\xrightarrow {q^{\otimes{(n+2)}}}\sfbar^n(B)\xrightarrow{\tau_B^n}\sfq^n\xrightarrow{f}B\xrightarrow{p}C. \] Our claim is then a consequence of the fact that $p(\delta)=0$ which implies that $p\xi_B^i=0$ for $i=1,2$.
\end{example}

\paragraph{\textbf{Acknowledgements.}} The first author gratefully acknowledges partial support from the NSERC of Canada, the FRQ-NT of Qu\'ebec and the Universit\'e de Sherbrooke. The second author is grateful to  the Universit\'e de Sherbrooke and  the University of Connecticut for their support. The third author gratefully acknowledges support from the NSF CAREER grant DMS-1254567 and the University of Connecticut. The fourth author is grateful to the UMI-CRM of the CNRS in Montr\'eal and  the Universit\'e de Sherbrooke for their support, enabling her to participate in this research project.

\end{document}